\documentclass[a4paper,11pt]{amsart}
\pdfoutput=1
\usepackage[T1]{fontenc}

\usepackage{lmodern}

\usepackage{amsmath,amssymb,amsthm}

\usepackage[pdftitle={Intrinsic colocal weak derivatives and Sobolev spaces between manifolds},pdfauthor={Alexandra Convent and Jean Van Schaftingen}]{hyperref}

\usepackage[abbrev]{amsrefs}
\usepackage[top=3cm,bottom=3cm]{geometry}
\usepackage{microtype}

\usepackage{enumerate}

\newtheorem{thm}{Theorem}[section]
\newtheorem{propo}[thm]{Proposition}
\newtheorem{lemme}[thm]{Lemma}

\theoremstyle{definition}
\newtheorem{de}{Definition}[section]

\theoremstyle{remark}

\DeclareMathOperator{\supp}{supp}
\DeclareMathOperator{\im}{Im}
\DeclareMathOperator{\id}{id}
\DeclareMathOperator{\Lip}{Lip}
\DeclareMathOperator{\Hom}{Hom}
\DeclareMathOperator{\Lin}{\mathcal{L}}
\DeclareMathOperator{\Vertlift}{Vert}

\newcommand{\Cc}{C^1_c(N,\R)}

\newcommand{\Cd}{C^1(N,\R)}
\newcommand{\Lipk}{\mathrm{Lip}(N,\R^k)}

\newcommand{\Lr}{L^p(M)}

\newcommand{\Wp}{\dot{W}^{1,p}(M,N)}

\newcommand{\dW}{\delta_{1,p}}
\newcommand{\dTN}{d}
\newcommand{\dM}{d_M}
\newcommand{\dE}{d_E}

\newcommand{\gF}{g^*_M\otimes g_N}
\newcommand{\gFk}{g^*_M\otimes g_k}
\providecommand{\psh}[2]{\ensuremath{\langle #1,#2\rangle}}

\providecommand{\abs}[1]{\lvert#1\rvert}
\newcommand{\bigabs}[1]{\bigl\lvert #1 \bigr\rvert}
\newcommand{\Bigabs}[1]{\Bigl\lvert #1 \Bigr\rvert}

\providecommand{\norm}[1]{\left\lVert#1\right\rVert}

\newcommand{\R}{\mathbb{R}}
\newcommand{\N}{\mathbb{N}}
\usepackage[all]{xy}

\title{Intrinsic colocal weak derivatives and Sobolev spaces between manifolds}
\author{Alexandra Convent}
\address{
  Institut de Recherche en Math\'ematique et Physique\\
  Universit{\'e} catholique de Louvain\\
  Chemin du Cyclotron 2 bte L7.01.01, 1348 Louvain-la-Neuve, Belgium}
\email{Alexandra.Convent@uclouvain.be}

\author{Jean Van Schaftingen } %\footnote{}
\address{
  Institut de Recherche en Math\'ematique et Physique\\
  Universit{\'e} catholique de Louvain\\
  Chemin du Cyclotron 2 bte L7.01.01, 1348 Louvain-la-Neuve, Belgium}
\email{Jean.VanSchaftingen@uclouvain.be}

\subjclass[2010]{58D15 (46E35, 46T10, 53C25, 58E20)}
\keywords{Sobolev space; weak derivative; approximate differentiability; truncation; weak convergence; strong convergence; metric space; complete distance; Riemannian manifold; natural metric on the tangent bundle; Sasaki metric; Cheeger--Gromoll metric; concordant distance; concordant metric.}

\begin{document}

%ABSTRACT
\begin{abstract}
We define the notion of colocally weakly differentiable maps from a manifold \(M\) to a manifold \(N\). If \(p \ge 1\) and \(M\) and \(N\) are endowed with a Riemannian metric, this allows us to define intrinsically the homogeneous Sobolev space \(\dot{W}^{1, p} (M, N)\). This new definition is equivalent with the definition by embedding in the Euclidean space and with the definition of Sobolev maps into a metric space. The colocal weak derivative is an approximate derivative. The colocal weak differentiability is stable under the suitable weak convergence. The Sobolev spaces can be endowed with various intrinsinc distances that induce the same topology and for which the space is complete.
\end{abstract}

\maketitle

\tableofcontents

%%%%% INTRO -------------------------------------------------------------
\section*{Introduction}
 
Sobolev spaces between manifolds are a natural tool to study variational problems for maps between manifolds, arising in geometry  \citelist{\cite{eels_lemaire}\cite{urakawa}\cite{linwang}} or in nonlinear physical models \citelist{\cite{ball}\cite{hardt}\cite{breziscoronlieb}}.
A Sobolev space of maps between the manifolds \(M\) and \(N\) can be defined by \citelist{\cite{hang}\cite{bethuel}\cite{dacorognafonsecamalytrivisa}\cite{breziscoronlieb}\cite{giaquintamodicasoucek}\cite{giaquintamodicasoucek2}\cite{giaquintamodicasoucek1989}\cite{giaquintamodicasoucek1990}\cite{hardt}\cite{mucci}\cite{ball}}
\begin{equation}
\label{def_first}
 \dot{W}^{1,p}(M,N) = \{ u : M \to N : 
 \text{\(\iota \circ u \in W^{1, 1}_{\mathrm{loc}} (M,\R^\nu)\) and \(\abs{D(\iota \circ u)}\in L^p (M)\)}\}.
\end{equation}
where \(\iota : N \to \R^\nu\) is an isometric embedding of the target manifold \(N\) in the Euclidean space \(\R^\nu\).
This definition is always possible, since every Riemannian manifold is isometrically embedded in a Euclidean space \citelist{\cite{nash1954}*{theorem 2}\cite{nash1956}*{theorem 3}}.

Since the embedding \(\iota : N \to \R^\nu\) is not unique, this definition could in principle depend on the choice of the embedding \(\iota\). 
In the particular case where \(\iota_1 : N \to \R^{\nu_1}\) and \(\iota_2 : N \to \R^{\nu_2}\) are embeddings and \(\iota_2 \circ \iota_1^{-1}\) is a bi-Lipschitz map for the induced Euclidean distances --- which is always the case if \(N\) is compact --- it can be seen directly that \(\dot{W}^{1, p} (M, N)\) and its topology are the same.

Such difficulties can be avoided by defining Sobolev spaces into \(N\) by using only the metric structure of \(N\), either by composition with Lipschitz maps \citelist{\cite{ambro}\cite{resh_97}} or by oscillations on balls \citelist{\cite{korevaar}\cite{jost}}. 
These definitions are equivalent to each other \citelist{\cite{jost}\cite{chiron}} and equivalent with the definition by isometric embedding \eqref{def_first} \cite{hajlasz_tyson}*{theorem 2.17}.
They are all intrinsic but they do not have any notion of weak derivative; they only provide a notion of Dirichlet integrand \(\abs{Du}\) which can differ from the one given by isometric embedding and depends on the integrability exponent \(p\) \cite{chiron}.
If \(N\) has a Riemannian structure then one can construct a posteriori an approximate derivative almost everywhere \cite{focardi_spadaro}; in contrast with the classical theory of Sobolev spaces between Euclidean spaces the derivative is a fine property of a function that plays no role in the definition of the Sobolev maps.
Several distances have been proposed for spaces of Sobolev maps between metric spaces,
but the spaces are not complete for any of these distances \cite{chiron}.

\medbreak

The goal of this work is to propose a robust intrinsic definition of Sobolev maps between manifolds in which the weak derivative plays a central role and to endow and well-behaved intrinsic metrics on the space of Sobolev maps.
We shall proceed in three steps: first we shall define a notion of differentiability and derivative, then we shall study the integrability of the derivative and finally we shall endow these spaces with convergence and metrics.
Each of these steps will require additional structure on the manifolds: at the beginning we shall simply use the differentiable structure of the manifolds, then a Riemannian metric on the manifolds and finally a Riemannian metric on their tangent bundles.
Defining the derivative before the space gives immediately the independence of the derivative from the Riemannian metric or the integrability exponent \(p\).
The primary role of the derivative in our approach will be quite handy to define complete intrinsic metrics. 

In the first step we define \emph{colocally weakly differentiable maps} as maps for which \(f \circ u\) is weakly differentiable when \(f \in \Cc\) (definition~\ref{cwdiff}).
The \emph{colocal weak derivative} is defined as the unique morphism of bundles \(Du\) such that the chain rule \(D(f\circ u) = Df \circ Du\) holds (definition~\ref{cwder}):
\[
   \xymatrix{
    TM \ar[r]^{D u} \ar[rd]_{D(f \circ u)} 
    & TN \ar[d]^{Df} \\
     & \R.
  }
\]
The colocal weak derivative has the usual nonlinear properties of a weak derivative; the definition extends previous definitions of the derivatives by truncation \cite{bbggpv}.
The colocal weak derivative is an approximate derivative (proposition~\ref{approx_der}). This follows from the Euclidean counterpart. We recover thus without fine properties of differentiable functions nor any Riemannian structure the derivative of Focardi and Spadaro \cite{focardi_spadaro}.

In the second step, we define when \(M\) and \(N\) are Riemannian manifolds, for every \(p \in [1,\infty]\), the \emph{homogeneous Sobolev space} (definition~\ref{hom_sobolev})
\begin{multline*}
  \dot{W}^{1,p}(M,N) = \{ u \colon M \to N \colon \text{\(u\) is colocally weakly differentiable}\\\text{ and }\abs{Du}_{\gF} \in L^p(M)\},
\end{multline*}
where the Euclidean norm \(\abs{\cdot}_{\gF}\) is induced by the Riemannian metrics on \(M\) and \(N\).
This definition is equivalent with  \eqref{def_first} when \(N\) is isometrically embedded in \(\R^\nu\) (proposition~\ref{equi_isometric}) --- \eqref{def_first} is thus a posteriori an intrinsic definition --- and with the definition of Sobolev spaces into metric spaces (proposition~\ref{thm_equi}). 
Given a colocally weakly differentiable map \(u\colon M \to N\), we characterize the quantity \(\abs{Du}_{\gF}\) as the smallest measurable function \(w \colon M \to \R\) such that for every \(f\in C^1_c(N, \R^{\min(\mathrm{dim}(M), \mathrm{dim}(N))})\),
\begin{equation}\label{NB_metric}
 \abs{D(f\circ u)} \le \abs{f}_{\Lip} \, w \qquad \text{ almost everywhere in } M.
\end{equation}
This allows to define a robust Dirichlet integrand; previous definitions with scalar test function \(f\in C^1_c(N,\R)\) were quite unstable \citelist{\cite{ambro}\cite{resh_97}}. 
Furthermore the inequality \eqref{NB_metric} might provide a robust definition of the Dirichlet integrand for Sobolev maps into metric spaces.

In a third step, we study sequences of colocally weakly differentiable maps (section~\ref{sectconvCWDM}) for which we prove a closure and a compactness property.

In the last step, we first study weakly converging sequences in the Sobolev space. Next, if \(d\) is a distance on the bundle of morphisms from \(TM\) to \(TN\) that satisfies some growth assumptions --- in particular \(d\) could be the distance induced by the Sasaki \cite{sasaki} and Cheeger-Gromoll \cite{cheeger_gromoll} metrics, or by an embedding in a Euclidean space --- we study the distance defined for all \(u,v\in W^{1,p}(M,N)\) as 
\[ \delta_{1,p}(u,v) = \left(\int_M d(Du,Dv)^p \right)^\frac{1}{p}. 
\]
If the Riemannian manifold \(N\) is complete, the Sobolev space is complete under all those distances, but those distances are not equivalent in general (proposition~\ref{not_equivalent}); these distances give the same convergent sequences than embedding in Euclidean spaces and than the noncomplete distances proposed on metric spaces \cite{chiron}.
In particular, we observe that the convergence induced by the natural distance arising from definition~\eqref{def_first} does not depend on the embedding.

%%% WEAK DERIVATIVE ------------------------------------------------
\numberwithin{equation}{section}

\section{Colocally weakly differentiable maps and colocal weak derivative}

\subsection{Weak differentiability on a differentiable manifold}
We assume that \(M\) and \(N\) are differentiable manifolds of dimensions \(m\) and \(n\) which are Hausdorff and have a countable basis \citelist{\cite{docarmo}*{\S 0.5}\cite{hirsch}*{\S 1.5}}.

We recall various definitions of local measure-theoretical notions on a manifold. 
A set \(E \subset M\) is \emph{negligible} if for every \(x \in M\) there exists a local chart \(\psi : V \subseteq M  \to \R^m\) -- that is \(\psi : V \subseteq M \to \psi (V) \subseteq \R^m\) is a diffeomorphism -- such that \(x \in V\) and the set \(\psi (E \cap V) \subset \R^m\) is negligible.
A map \(u : M \to N\) is \emph{measurable} if for every \(x \in M\) there exists a local chart \(\psi : V \subseteq M  \to \R^m\) such that \(x \in V\) and the map \(u \circ \psi^{-1}\) is measurable \citelist{\cite{hirsch}*{\S 3.1}\cite{derham}*{\S 3}}.
A function \(u : M \to \R\) is \emph{locally integrable} if for every \(x \in M\) there exists a local chart \(\psi : V \subseteq M \to \R^m\) such that \(x \in M\) and \(u \circ \psi^{-1}\) is integrable on \(\psi (V)\) \cite{hormander}*{\S 6.3}.
Similarly, a locally integrable map \(u : M \to \R\) is \emph{weakly differentiable} if for every \(x \in M\) there exists a local chart \(\psi : V \subseteq M  \to \R^m\) such that \(x \in V\) and the map \(u \circ \psi^{-1}\) is weakly differentiable. 
All these notions are independent on any particular metric or measure on the manifold \(M\).

A Radon measure \(\mu\) on \(M\) is \emph{absolutely continuous} if for every \(x \in M\) there exists a local chart \(\psi : V \subseteq M  \to \R^m\) such that \(x \in V\) and the image measure \(\psi_* (\mu)\) defined by \(\psi_* (\mu) (A) = \mu (\psi^{-1} (A))\) is absolutely continuous with respect to the Lebesgue measure. 
The measure \(\mu\) is \emph{positive} if for every \(x \in M\) there exists a local chart \(\psi : V \subseteq M  \to \R^m\) such that \(x \in V\) and the Lebesgue measure on \(\psi (V)\) is absolutely continuous with respect to \(\psi_* (\mu) (A) = \mu (\psi^{-1} (A))\) \cite{katokhasselblatt}*{definition 5.1.1}. Every absolutely continuous measure on \(M\) is absolutely continuous with respect to every positive measure.
A set \(E \subset M\) is negligible if and only if \(\mu (E) = 0\) for every positive absolutely continuous Radon measure \(\mu\). A measurable function \(u : M \to \R\) is locally integrable if and only if there exists a positive absolutely continuous Radon measure \(\mu\) such that \(\int_{M} \abs{u} \, d \mu < \infty\).

\medbreak

We first define the notion of colocally weakly differentiable map.

%DEF WEAKLY DIFF
\begin{de}\label{cwdiff}
A map \(u \colon M \to N\) is \emph{colocally weakly differentiable} if \(u\) is measurable and for every \( f \in \Cc\), \(f \circ u \) is weakly differentiable.  
\end{de}

When \(N=\R\), the space of colocally weakly differentiable functions is the space of Sobolev functions by truncations \(\mathcal{T}^{1, 1}_{\mathrm{loc}} (M)\) \cite{bbggpv}.
If \(u : \R^m \to \R^n\) is weakly differentiable, then \(u\) is colocally weakly differentiable. The converse is false: for example, the function \(u\colon \R^m \to \R\) defined for every \(x \in \R^m \setminus \{0\}\) by \(u(x) = \abs{x}^{-\alpha}\) is not weakly differentiable for any \(\alpha > m - 1\), but is colocally weakly differentiable for every \(\alpha \in \R\). 

In order to define the colocal weak derivative, we denote by \((TM, \pi_M, M)\) the \emph{tangent bundle} over \(M\), that is,
\[
  TM= \bigcup_{x\in M} \{x\} \times T_x M ,
\]
\(\pi_M \colon TM \to M\) is the natural projection and for every \(x\in M\), the fiber \(\pi_M^{-1} (\{x\})\) is isomorphic to \(\R^m\);
a map \( \upsilon \colon TM  \to TN \) is a \emph{bundle morphism} that covers \(u \colon M \to N\) if 
\[
 \xymatrix{
    TM \ar[r]^\upsilon \ar[d]_{\pi_M} & TN \ar[d]^{\pi_N} \\
    M \ar[r]_u & N
  }
\]
commutes, that is, \(\pi_N \circ \upsilon = u \circ \pi_M\), and for every \(x \in M\), \(\upsilon(x) \colon \pi^{-1}_M(\{x\}) \to \pi^{-1}_N(\{u(x)\})\) is linear.
The space of all bundle morphisms  is denoted by \(\Hom(TM,TN)\).
In particular, if \(u\colon M \to N\) is a differentiable map, then \(Du \colon TM \to TN\) is a bundle morphism that covers \(u\). 
By a direct covering argument, if \(u : M \to \R\) is weakly differentiable, then there exists a bundle morphism \(Du : TM \to \R\) such that for every local chart \(\psi : V \subseteq M \to \R^m\), 
\[
  D (u \circ \psi^{-1}) = Du \circ D \psi^{-1}
\]
almost everywhere on \(\psi (V)\).
We have now all the ingredients to define the colocal weak derivative.

%DEF weakly differentiable 
\begin{de}\label{cwder}
Let \(u\colon M \to N\) be a colocally weakly differentiable map. A map \(Du\colon TM \to TN\) is a \emph{colocal weak derivative} of \(u\) if \(Du\) is a measurable bundle morphism that covers \(u\) 
%that is, \(Du \colon M \to \Lin(TM,TN)\) is measurable, 
and for every \(f\in \Cc\),  
\[D(f\circ u ) = Df \circ Du \] 
almost everywhere in \(M\).
\end{de}
Consequently, if \(Du\) is a colocal weak derivative of \(u\), for almost every \(x\in M\), \(Du(x) \in \Lin(T_x M, T_{u(x)}N)\) and for each \(e \in T_x M\), \(D(f\circ u)(x)[e]= Df(u(x))[Du(x)[e]]\).

We first observe that this notion extends the notion of classical differentiability:

\begin{propo}[Equivalence of classical and colocal weak derivatives]
Let \(u \in C(M,N)\). Then \(u\) has a continuous colocal weak derivative if and only if  \(u\in C^1(M,N)\).
Moreover, the colocal weak derivative and the classical derivative coincide almost everywhere.
\end{propo}

\begin{proof}
Since for every \(f\in \Cc\), \(f\circ u\) is weakly differentiable, we can apply the equivalence of classical and weak derivatives (Du Bois-Reymond lemma) \citelist{\cite{willem_en}*{theorem 6.1.4}\cite{lieb_loss}*{theorem 6.10}} and local charts on \(M\) to obtain that \(f\circ u \in C^1(M,\R)\). Since \(f\) is arbitrary, the map \(u\) is continuously differentiable. 
\end{proof}

\begin{de}
A bundle morphism \( \upsilon : TM \to TN\) that covers \(u : M \to N\) is \emph{bilocally integrable} on \(A \subseteq M\) if there exist local charts \(\psi : V \subseteq M \to \R^m\), \(\varphi : U \subseteq N \to \R^n\) such that if \(L \subset V\) and \(K \subset U\) are compact, then the function
\(D\varphi \circ \upsilon_{\vert V} \circ D(\psi^{-1})\) is integrable on \(\psi( A \cap L \cap u^{-1} (K))\).
\end{de}

If \(\mu\) is a positive absolutely continuous measure on \(M\), the morphism \(\upsilon\) is bilocally integrable if and only if there exists a continuous norm \(\abs{\cdot}\) on \(T^*M \otimes TN\) such that \(\int_{M} \abs{\upsilon}\,d\mu < \infty\).

The main result of the current section is that colocally weakly differentiable maps have a colocal weak derivative.

%THM WEAK DIFF
\begin{propo}[Existence and uniqueness of the colocal weak derivative]\label{thm_weak}
If the map \(u\colon M \to N\) is colocally weakly differentiable, then \(u\) has a unique colocal weak derivative \(Du \colon TM \to TN\). Moreover, the bundle morphism \(Du\) is bilocally integrable and for every \(f\in C^1(N,\R)\) for which \(f\circ u : M \to \R\) is weakly differentiable,
\begin{equation*}
 D(f\circ u ) = D f \circ D u.
\end{equation*}
\end{propo}

This result was already known when \(N = \R\) \cite{bbggpv}*{lemma 2.1}; as remarked there, the colocal weak derivative need not be locally integrable.
The important geometric tool for proving proposition~\ref{thm_weak} is the existence of extended local charts. This construction is reminiscent of the patch mappings to the sphere \(\varphi \in C^1 (N, \mathbb{S}^n)\) \cite{nash1956}*{p. 60}.

%LEMME U, varphi, varpi^*
\begin{lemme}[Extended local charts]\label{lemme_diffeo}
For every \(y \in N\), there exist an open subset \(U\subseteq N\) such that \(y\in U\), and maps \(\varphi \in C^1(N,\R^n)\) and \(\varphi^* \in C^1(\R^n,N)\) such that 
\begin{enumerate}[(i)]
  \item the set \(\supp \varphi\) is compact,
  \item the set \(\overline{\{ x \in \R^n \colon \varphi^*(x) \neq \varphi^*(\varphi(y))\}}\) is compact,
  \item the map \(\varphi_{\vert U}\) is a diffeomorphism onto its image \(\varphi(U)\),
  \item \(\varphi^* \circ \varphi= \id\) in \(U\).
\end{enumerate}
\end{lemme}

\begin{proof}
By definition of differentiable manifold, there exists a local chart \(\psi : V \subseteq N \to \psi(V) \subseteq \R^n\) around \(y \in N\). 
Without loss of generality, we assume that \(\psi(y) = 0\). 
Since the set \(\psi (V) \subseteq \R^n\) is open, there exists \(r>0\) such that \(B_{2 r} \subseteq \psi(V)\). We choose a map \(\theta \in C^1_c(\R^n,\R)\) such that \(0\le \theta \le 1\) on \(\R^n\), \(\theta = 1\) on \(B_r\) and \(\supp(\theta) \subset B_{2 r}\). 
We take the set \(U = (\psi_{\arrowvert V})^{-1}(B_r)\) and the maps \(\varphi \colon N \to \R^n\) defined for every \(z \in N\) by
\begin{equation*}
\varphi(z)=
\begin{cases}
\theta(\psi(z)) \psi(z) & \text{ if } z \in (\psi_{\arrowvert V})^{-1}(B_{2 r}), \\
0 & \text{ otherwise }\\
\end{cases}
\end{equation*}
and \(\varphi^*\colon \R^n \to N\) defined for every \(x\in \R^n\) by \(\varphi^*(x) = (\psi_{\arrowvert V})^{-1}(\theta(x) \,x).\)
\end{proof}

We begin by proving a local counterpart of proposition~\ref{thm_weak}. 

%LEMME FORMULA LOCALLY
\begin{lemme}\label{lemme_formula}
If \(u \colon M \to N\) is a colocally weakly differentiable map and \(y\in N\), then there exist an open subset \(U \subseteq N\) such that \(y\in U\) 
and a unique measurable bundle morphism  \(D_U u : \pi_M^{-1} (u^{-1} (U)) \to TN\) such that for every \(f\in C^1(N,\R)\) for which \(f\circ u : M \to \R\) is weakly differentiable,
\begin{equation*}\label{form_L}
 D(f\circ u ) = D f \circ D_U u \qquad \text{ almost everywhere on } u^{-1}(U).
\end{equation*}
Moreover, \(D_U u\) is bilocally integrable on \(u^{-1} (U)\).
\end{lemme}

\begin{proof}
Let \(U \subseteq N\), \(\varphi \in C^1(N,\R^n)\) and \(\varphi^* \in C^1(\R^n,N)\) be the extended local charts given by lemma~\ref{lemme_diffeo}.
Since \(u\) is colocally weakly differentiable, the map \(\varphi \circ u : M \to \R^n\) is weakly differentiable. 
Since for every \(x\in u^{-1}(U)\), the linear map between tangent spaces \(D\varphi(u(x)) \colon T_{u(x)}N \to \R^n\) is invertible, the map \(D_U u\) is uniquely defined for almost every \(x\in u^{-1}(U)\) by
\[ 
  D_U u(x)= (D\varphi(u(x))^{-1} \circ (D(\varphi \circ u )(x)). 
\]

If \(f\in \Cd\) and \(f\circ u : M \to \R\) is weakly differentiable, since \(\varphi^* (\R^n)\) is compact and \(f\circ \varphi^* \in C^1 (\R^n,\R)\), the chain rule for weakly differentiable functions implies (see for example \citelist{\cite{willem_en}*{theorem 6.1.13}\cite{evans_gariepy}*{theorem 4.2.4 (ii)}}) that \(f \circ \varphi^* \circ \varphi \circ u\) is weakly differentiable and 
\[
D(f \circ \varphi^* \circ \varphi \circ u) = D(f\circ \varphi^*) \circ D(\varphi \circ u) = D(f\circ \varphi^*)\circ D\varphi \circ D_U u
\]
Since \(f \circ u = f \circ \varphi^* \circ \varphi \circ u\) on \(u^{-1}(U)\) and \(D(f\circ \varphi^*)\circ D\varphi = Df\) on \(U\), we have \citelist{\cite{willem_en}*{corollary 6.1.14}\cite{evans_gariepy}*{theorem 4.2.4 (iv)}\cite{lieb_loss}*{theorem 6.19}} almost everywhere on \(u^{-1}(U)\),
\begin{equation*}
\begin{split}
D(f\circ u) = D(f\circ \varphi^*)\circ D\varphi \circ D_U u = Df \circ D_U u. \qedhere
\end{split}
\end{equation*}
\end{proof}

\begin{proof}[Proof of proposition~\ref{thm_weak}]
Let \((U_i)_{i\in I}\) be an open cover of \(N\) by sets given by lemma~\ref{lemme_formula}. 
Since the manifold \(N\) has a countable basis, we can assume that \(I\) is at most countable \cite{munkres}*{theorem 3.30}.
%compatibility
Let \(D_{U_i} u\) and \(D_{U_i \cap U_j} u\) be the derivatives given by lemma~\ref{lemme_formula}.
By uniqueness, \(D_{U_i\cap U_j} u = D_{U_i} u = D_{U_j} u\) almost everywhere on \(u^{-1}(U_i \cap U_j)\).
Since \(\bigcup_{i\in I} \, u^{-1}(U_i) = M\) and \(I\) is countable, the bundle morphism \(Du : TM \to TN\) can be defined uniquely almost everywhere by
\[ 
  Du = D_{U_i} u \, \text{ almost everywhere on } u^{-1}(U_i).\qedhere
\]
\end{proof} 

\subsection{Properties of the colocal weak derivative}
The colocal weak derivative retains some properties of weak derivatives.

\begin{propo}[Chain rule]
\label{propoChainColocal}
If the map \(u\colon M \to N\) is colocally weakly differentiable and \(f\in C^1 (N, \Tilde{N})\) such that \(f\circ u\) is colocally weakly differentiable, then
\begin{equation*}
 D(f\circ u ) = D f \circ Du \text{ almost everywhere in } M.
\end{equation*}
\end{propo}
\begin{proof}
For every \(\varphi \in C^1_c (\Tilde{N}, \R)\), \(\varphi \circ f \in C^1 (N, \R)\) and \(\varphi \circ f\) is weakly differentiable. Therefore, by proposition~\ref{thm_weak} and the classical chain rule,
\[
 D \varphi \circ D (f \circ u) = D (\varphi \circ f \circ u) = D \bigl((\varphi \circ f) \circ u\bigr) =  D(\varphi \circ f) \circ D u = D \varphi \circ (D f \circ D u).\qedhere
\]

\end{proof}

%EMBEGGING WEAKLY DIFFERENTIABLE
\begin{propo}
\label{propoEmbeddingWeakDerivative}
Let \( \iota \colon N \to \Tilde{N}\) be an embedding and let \(u\colon M \to N\) be a map. 
If \(\iota \circ u\) is colocally weakly differentiable, then \(u\) is colocally weakly differentiable and \(Du\) is the unique bundle morphism such that 
\[
  D (\iota \circ u) = D \iota \circ D u
\]
almost everywhere on \(M\).\\
If moreover \(\iota(N)\) is closed, then the converse holds. 
\end{propo}
\begin{proof}
Since \(\iota\) is an embedding, \(\iota (N)\) has a tubular neighborhood in \(N\): there exists a vector bundle \((E, \pi_N, N)\) and an embedding \(\Tilde{\iota} : E \to \Tilde{N}\) such that \(\Tilde{\iota} \vert_{N} = \iota\) and \(\Tilde{\iota} (E)\) is open in \(\Tilde{N}\) \cite{hirsch}*{theorem 4.5.2}. In particular, \(\iota \circ \pi_N \circ \Tilde{\iota}^{-1}\) is a retraction of the tubular neighborhood \(\Tilde{\iota} (E)\) on \(\iota (N)\).
Let \(f\in C^1_c(N,\R)\). We choose \(\eta \in C^1_c (\Tilde{N}, \R)\) such that \(\eta = 1\) on \(\iota (\supp f)\) and \(\supp \eta \subset \Tilde{\iota} (E)\) and define \(\Tilde{f} = (f \circ \pi_N \circ \Tilde{\iota}^{-1})\eta  : \Tilde{N} \to \R\). Since \(\supp \eta \subset \Tilde{\iota}(E)\), the function \(\Tilde{f}\) is well-defined, \(\Tilde{f} \in C^1_c (\Tilde{N}, \R)\) and \(\Tilde{f} \circ \iota = f\) on \(N\). In particular, \(f \circ u =  \Tilde{f} \circ \iota \circ u\) is weakly differentiable.

Conversely, if \(\iota (N)\) is closed, then for every \(\varphi \in C^1_c (\Tilde{N}, \R)\), \(\varphi \circ \iota \in C^1_c (N, \R)\) and \(\iota \circ u\) is thus colocally weakly differentiable.
\end{proof}

By the Whitney embedding theorem \cite{whitney}, there always exists an embedding \(\iota : N \to \R^\nu = \R^{2n + 1}\) such that \(\iota (N)\) is closed \citelist{\cite{adachi}*{theorem 2.6}\cite{derham}*{theorem 5}}. Proposition~\ref{propoEmbeddingWeakDerivative} gives thus an equivalent definition of colocal weak differentiability; the drawback of this alternate approach to colocal weak differentiability is its dependence on the Whitney embedding theorem for differentiable manifolds.

Since differentiable manifold do not have in general any algebraic structure and since the colocally weakly differentiable functions between Euclidean spaces do not form a vector space \cite{bbggpv}, the colocal weak derivative does not have any algebraic properties of sum or product. There is however still a property of the cartesian product of maps.

%CLASSICAL PROPERTIES%PRODUCT
\begin{propo}[Product of manifolds]
Let \(N_1, N_2\) be two differentiable manifolds. 
If \(u_1 \colon M \to N_1\) and  \(u_2 \colon M \to N_2\) are colocally weakly differentiable, then the map \(u= (u_1,u_2) \colon M \to N_1\times N_2 \) is colocally weakly differentiable and 
\[ Du=(Du_1,Du_2) \qquad \text{ almost everywhere in } M.\]
\end{propo}

The uniqueness property can be refined for maps that coincide on a set of positive measure:

%EQUALITY OF DERIVATIVE 
\begin{propo}[Equality of derivatives]\label{propo_equ}
If the maps \(u, v \colon M \to N\) are colocally weakly differentiable, then \(Du = Dv\) almost everywhere on the set \(\{x \in M : u (x) = v (x)\}\). 
\end{propo}

\begin{proof}
Let \(A=\{x \in M : u (x) = v (x)\}\). Let \(U\subseteq N\) and \(\varphi \in C^1(N,\R^n)\) be the extended local chart given by lemma~\ref{lemme_diffeo} and let \(\psi : V \subseteq M \to \R^m\) be a local chart. 
Since \(\varphi \circ u \circ \psi^{-1}\) and \(\varphi\circ v \circ \psi^{-1}\) are weakly differentiable and \(\varphi \circ u \circ \psi^{-1}= \varphi \circ v \circ \psi^{-1}\) on \(\psi(A \cap V)\), 
\(D(\varphi \circ u \circ \psi^{-1}) = D(\varphi \circ v \circ \psi^{-1})\) almost everywhere on \(\psi (A \cap V)\). 
By definition of the colocal weak derivative, \(D\varphi \circ Du \circ D \psi^{-1}= D\varphi\circ Dv \circ D \psi^{-1}\) almost everywhere on \(\psi (A \cap V)\). 
Since \(D\varphi\) is invertible on \(u^{-1}(U)\),
\(Du = Dv\) almost everywhere on \(A \cap V\cap u^{-1}(U)\cap v^{-1}(U)\). Taking open countable covers \((U_i)_{i\in I}\) of \(N\) and \((V_j)_{j \in J}\) of \(M\), we conclude that \(Du = Dv\) almost everywhere on \(A\).
\end{proof}

%APPROXIMATE DIFFERENTIABLE
\subsection{Approximate differentiability}
We study the relationship between the colocal derivative and the approximate derivative.
For a map between differentiable manifolds, we generalize the classical definition of approximate derivative of maps between vector spaces \citelist{\cite{evans_gariepy}*{definition 6.1}\cite{federer}*{3.1.2}}.

\begin{de}
Let \(u \colon M \to N\) and \(x \in M\). 
The linear map \(\xi \colon T_x M \to T_{u (x)} N\) is an \emph{approximate derivative} of \(u\) at \(x\) if there exist local charts \(\psi: V \to \psi (V) \subseteq \R^m\)  with \(x \in V\) and \(V \subseteq M\) open, and \(\varphi: U \to \varphi (U) \subseteq \R^n\) with \(u (x) \in U\) and \(U \subseteq N\) open such that for every \(\varepsilon > 0\),
\begin{multline*}
  \lim_{\rho \to 0} \rho^{-m}\mathcal{L}^m \bigl(\psi(\{y \in V: u (y) \in U\\ 
  \abs{\varphi (u(y)) - \varphi(u(x)) - (D \varphi (u(x)) \circ \xi \circ \bigl(D\psi (x)\bigr)^{-1})[\psi (y) - \psi (x)] }\\
  > \varepsilon \abs{\psi (y) - \psi (x)}  \}) \cap B_\rho (\psi (x))\bigr)
  = 0.
\end{multline*}
\end{de}

The approximate derivative is unique and it is sufficient to establish its existence for one pair of diffeomorphisms.

If \(M\) and \(N\) are endowed with Riemannian metrics \(g_M\) and \( g_N\) respectively, it is natural to take for \(\varphi\) and \(\psi\) the exponential coordinates and to use the Riemannian distances \(d_N\) and \(d_M\) and measure \(\mu_M\); the approximate differentiability can be observed to be equivalent with requiring for every \(\varepsilon > 0\)
\[
  \lim_{\rho \to 0} \rho^{-m} \mu_M \bigl(\{y \in B^M_\rho (x): d_N (u (x),\exp_{u (x)} ( \xi (\exp^{-1}_{x} (y))) > \varepsilon d_M (x, y) \}\bigr) = 0;
\]
we recover thus in this particular case the definition of Focardi and Spadaro for maps from the Euclidean space to a Riemannian manifold \cite{focardi_spadaro}*{definition 0.3}.

%COLOCALLY -> APPROXIMATE 
\begin{propo}[Approximate differentiability of colocally weakly differentiable maps]\label{approx_der}
If \(u \colon M \to N\) is colocally weakly differentiable, then for almost every \(x \in M\), the colocal weak derivative \(D u (x)\) is the approximate derivative of \(u\) at \(x\).
\end{propo}
\begin{proof}
Let \(\psi : V \subseteq M \to \R^m\) be a local chart around \(x \in M\). Let \( U \subseteq N\) and \(\varphi \in C^1(N,\R^n)\) be the extended local chart given by lemma~\ref{lemme_diffeo}. 
Since \(u\) is colocally weakly differentiable, \(\varphi \circ u \circ \psi^{-1}: \psi (V) \to  \R^n\) is weakly differentiable and \(D (\varphi \circ u \circ \psi^{-1})
=D \varphi \circ Du \circ D \psi^{-1}\) on \(\psi (V)\).
Since weakly differentiable maps between vector spaces are approximately differentiable \cite{evans_gariepy}*{theorem 6.1.4}, \(\varphi \circ u \circ \psi^{-1}\) is approximately differentiable almost everywhere on \(\psi (V)\), that is, 
\begin{multline*}
  \lim_{\rho \to 0} \rho^{-m}\mathcal{L}^m \bigl(\psi(\{y \in V: \\
  \abs{\varphi (u(y)) - \varphi(u(x)) - (D \varphi (u(x)) \circ D u (x) \circ \bigl(D\psi (x)\bigr)^{-1})[\psi (y) - \psi (x)] }\\
  > \varepsilon \abs{\psi (y) - \psi (x)}  \})\cap B_\rho (\psi (x))\bigr)
  = 0.
\end{multline*}
Next, we note that since \(u\) is measurable, almost every \(x \in M\) is a Lebesgue point of \(u\). 
Since the set \(U\) is open, for almost every \(x \in u^{-1} (U)\), 
\[
  \lim_{\rho \to 0} \rho^{-m}\mathcal{L}^m \bigl(\psi(\{y \in V: u (y) \not \in U \})\cap B_\rho (\psi (x))\bigr)
  = 0.
\]
Therefore we have for almost every \(x \in u^{-1} (U)\),
\begin{multline*}
  \lim_{\rho \to 0} \rho^{-m}\mathcal{L}^m \bigl(\psi(\{y \in V: u (y) \in U\\ 
  \abs{\varphi_{\vert U} (u(y)) - \varphi_{\vert U} (u(x)) - (D \varphi (u(x)) \circ Du(x) \circ \bigl(D\psi (x)\bigr)^{-1})[\psi (y) - \psi (x)] }\\
  > \varepsilon \abs{\psi (y) - \psi (x)}  \})\cap B_\rho (\psi (x))\bigr)
  = 0,
\end{multline*}
that is, \(Du (x)\) is the approximate derivative of \(u\) for almost every \(x \in u^{-1} (U)\). The conclusion follows by a countable covering argument.
\end{proof}

%%% SOBOLEV DEFINITION -------------------------------------------------
\section{Sobolev maps between Riemannian manifolds}

\paragraph{Preliminaries}
We assume now that \((M,g_M)\) and \((N,g_N)\) are Riemannian manifolds.
%DEF "FROBENIUS"
In particular the metrics on vectors of \(TM\) and \(TN\) induce a metric \(\gF\) on \(T^* M \otimes TN\). This metric can be computed for every \(\xi \in T^*_xM \otimes T_yN = \Lin (T_x M, T_y N)\) by
\[ 
  (\gF)(\xi,\xi) = \sum_{i=1}^m g_N \bigl(\xi(e_i), \xi(e_i)\bigr),
\]
where \((e_i)_{1\le i \le m}\) is an orthonormal basis in \(\pi^{-1}_M(\{x\})\) with respect to the Riemannian metric \(g_M\). %This definition is independent of the choice of the basis. 

We are now in position to define the Sobolev spaces.

\begin{de}\label{hom_sobolev}
Let \(p \in [1,\infty)\). 
A map \(u \colon M \to N\) belongs to the \emph{Sobolev space} \(\Wp\) if \(u\) is colocally weakly differentiable and \(\abs{Du}_{\gF} \in \Lr\).
\end{de}

When \(N\) is a Euclidean space, this space coincides with the classical homogeneous Sobolev space \cite{bbggpv}*{lemma 2.1}. %If \(p= \infty\), \(\Wp = \Lip(M, \R)\). 

\paragraph{Characterization of the norm of the derivative}
We characterize the quantity \(\abs{Du}_{\gF}\) that appears in the definition of Sobolev spaces. 
We recall that the \emph{operator norm} is defined for every \( \xi \in T^*_x M \otimes T_y N =\Lin(T_x M, T_y N) = \Hom (TM,TN)_{x,y}\) by
\[
  \abs{\xi}_{\Lin} = \sup \{\abs{\xi(e)} \colon e \in T_y N, \, \abs{e}_{g_N} \le 1\};
\] 
the \emph{Lipschitz semi-norm} of \(f \colon N \to \R^k\) is defined by 
\[ 
  \abs{f}_{\Lip} = \sup\left\{ \frac{\abs{f(y) - f(z)}}{d_N(y, z)} \colon y,z \in N \text{ and } y \ne z \right\}
\]
where \(d_N\) is the distance on \(N\) induced by the Riemannian metric \(g_N\). For every \(k\ge 1\), we denote by \(g_k\) the standard Euclidean metric on \(\R^k\). 

%EQUIVALENT DEF W^1,1_loc
\begin{propo}\label{thm_equi}
Let \(k \ge \min(m,n)\). Let \(u\colon M \to N \), \(w \colon M \to \R\) be measurable maps. The following statements are equivalent. 
\begin{enumerate}[(i)]
  \item \label{thm_equi1} \(u\) is colocally weakly differentiable and \(\abs{Du}_{\gF} \le w\) almost everywhere in \(M\), 
  \item \label{thm_equi2} for every \(f\in C^1_c(N,\R^k)\), \(f\circ u\) is weakly differentiable  and almost everywhere in \(M\)
    \begin{equation*}
      \abs{D(f\circ u)}_{\gFk} \le \abs{Df(u)}_{\Lin} w,
    \end{equation*}
  \item \label{thm_equi3} for every \(f\in C^1_c(N,\R^k)\), \(f\circ u\) is weakly differentiable  and almost everywhere in \(M\)
    \begin{equation*}
      \abs{D(f\circ u)}_{\gFk} \le \abs{f}_{\Lip}\, w,
    \end{equation*}
\end{enumerate}
If moreover \(w \in L^1_\mathrm{loc} (M)\), then for every \(f\in \Lipk \), \(f\circ u\) is weakly differentiable and almost everywhere in \(M\)
  \begin{equation*}
    \abs{D(f\circ u )}_{\gFk} \le \abs{f}_{\Lip}\, w. 
  \end{equation*}
\end{propo}

Since the first assertion is independent of \(k\), there is also equivalence between these statements for every \(k \ge \min(m,n)\).  

Proposition~\ref{thm_equi} implies in particular that if \(p \in [1, \infty)\) and \(M\) is an open bounded subset of \(\R^m\), definition~\ref{hom_sobolev} is equivalent with the notion of Sobolev spaces into metric spaces \cite{resh_97}*{theorem 5.1}. 

% APPROXIMATION LIP MANIFOLDS
In order to prove proposition~\ref{thm_equi}, we shall use an approximation property of Lipschitz maps on manifolds.
\begin{lemme}[Approximation of Lipschitz maps]
\label{lemme_approx}
Let \(f \in \Lip (N, \R^k)\). 
There exists a sequence \((f_\ell)_{\ell \in \N}\) of maps in \(C^1_c (N, \R^k)\) that converges uniformly over compact subsets of \(N\) and  such that
\[
  \limsup_{\ell \to \infty} \abs{f_\ell}_{\Lip} \le \abs{f}_{\Lip}.
\]
\end{lemme}

\begin{proof}[Sketch of the proof]
Given \(y \in N\) and \(\theta \in C^1_c ([0, \infty), \R)\) such that \(\theta = 1\) on \([0, 1]\), we define for every \(\ell \in \N_*\),
\[
  \theta_\ell(z) = \theta\Bigl(\sqrt[\ell]{d_N(y,z)}\Bigr), 
\]
and observe that \(\theta_\ell \in C^1_c  (N, \R)\), 
\[
  \abs{D \theta_\ell}_{\Lin} \le \frac{\norm{\theta'}_{L^\infty}}{\ell},
\]
and \((\theta_\ell)_{\ell \in \N_*}\) converges to \(1\) uniformly over compact subsets of \(N\).
We also define \(T_\ell : \R^k \to \R^k\) for \(\ell \in \N_*\) and \(x \in \R^k\) by 
\[
  T_\ell (x) =
  \begin{cases}
    x & \text{if \(\abs{x} \le \ell\)},\\
    \ell x/\abs{x} & \text{if \(\abs{x} > \ell\)}.
  \end{cases}
\]
For every \(\ell \in \N_*\), \(T_\ell\) is nonexpansive and bounded by \(\ell\) and the sequence \((T_\ell)_{\ell \in \N_*}\) converges uniformly to the identity over compact subsets. 

If we define \(\Tilde{f}_\ell = \theta_\ell \cdot (T_{\ell} \circ f)\),
we observe for every \(\ell \in \N_*\) that \(\Tilde{f}_\ell \in \Lip (N, \R^k)\), that  
\[
  \abs{\Tilde{f}_\ell}_{\Lip} \le \abs{f}_{\Lip} + \frac{\norm{\theta'}_{L^\infty}}{\ell},
\] 
and that the support of \(f_\ell\) is compact; the sequence \((\Tilde{f}_\ell)_{\ell \in \N_*}\) converges uniformly over compact subsets. 
Hence the conclusion follows by approximating uniformly \(\Tilde{f}_\ell\) by differentiable functions with a control on the Lipschitz norm \citelist{\cite{greene_wu}*{lemma 8}\cite{aflr}}. 
\end{proof}

We shall also rely on a refined version of the extended local charts of lemma~\ref{lemme_diffeo}. 

%LEMME MEILLEUR CONTROLE NORME DERIVEE
\begin{lemme}[Almost isometric extended local charts]\label{lemme_iso}
For every \(y \in N\) and every \(\varepsilon > 0\), there exist an open subset \(U\subseteq N\) such that \(y\in U\),
and maps \(\varphi \in C^1(N,\R^n)\) and \(\varphi^* \in C^1(\R^n,N)\) such that 
\begin{enumerate}[(i)]
  \item the set \(\supp \varphi\) is compact,
  \item the set \(\overline{\{ x \in \R^n \colon \varphi^*(x) \neq \varphi^*(\varphi(y))\}}\) is compact,
  \item the map \(\varphi_{\vert U}\) is a diffeomorphism onto its image \(\varphi(U)\),
  \item \(\varphi^* \circ \varphi= \id\) in \(U\),
  \item \label{propControl} for every \(z \in N\),
\begin{align*}
&\abs{D\varphi(z)}_{\Lin} \le 1+\varepsilon &
\text{ and }&
&\abs{D\varphi^*(\varphi(z))}_{\Lin} \le 1+\varepsilon.
\end{align*}
\end{enumerate} 
\end{lemme}
The difference with lemma~\ref{lemme_diffeo} lies in the control \eqref{propControl} on the operator norms of \(D\varphi\) and \(D\varphi^*\). 

\begin{proof}[Proof of lemma~\ref{lemme_iso}]
Since \(N\) is a differentiable manifold, there exists a local chart \(\psi : V \subseteq N \to \R^n\) around \(y\in N\). Up to an affine transformation on \(\R^n\) we can assume that \(\psi (y) = 0\) and \(D\psi(y)\in \mathcal{L}(T_y N, \R^n)\) is an isometry.

By continuity of \(D \psi\), there exists \(\delta>0\) such that if \(z \in N\) and  \(d_N(y,z) \le \delta\), then 
\begin{align*} 
\abs{D\psi(z)}_{\Lin} &\le 1 + \varepsilon &
&\text{ and } &
\bigabs{D(\psi_{\arrowvert_V})^{-1}(\psi(z))}_{\Lin} & \le 1 + \varepsilon.
\end{align*}
We choose \(r>0\) such that \(B_{3 r} \subseteq \psi (B_\delta (y))\). 
We take a map \(\theta \in C^1_c(\R^n,\R^n)\) such that \(\theta = \id\) on \(B_r\), \(\supp(\theta) \subset B_{3 r}\) and \(\sup_{x\in \R^n} \abs{D\theta(x)}_{\Lin} \le 1\). We finally define \(\varphi = \theta \circ \psi\), \( \varphi^* = (\psi_{\arrowvert V})^{-1}\circ \theta\) and \(U= (\psi_{\arrowvert V})^{-1}(B_r)\) and we conclude as in the proof of lemma~\ref{lemme_diffeo}.
\end{proof}

We shall also use the following lemma to compute Euclidean norms of maps.
%LEMMA FROBENIUS NORM
\begin{lemme}[Reduction of the Euclidean norm of operators]\label{frob_norm}
Let \(x\in M\) and let \(\xi \in \Lin(T_x M,\R^n)\) be a linear map. If \(k \ge \min (m, n)\), then 
\[\abs{\xi}_{g^*_M\otimes g_n} = \sup \bigl\{ \abs{ \rho \circ \xi}_{g^*_M \otimes g_k}: \rho \in \Lin (\R^n,\R^k) \text{ and } \abs{\rho}_{\Lin} \le 1\bigr\}.
\]
\end{lemme}

\begin{proof}
On the one hand, if \(\rho \in \Lin (\R^n,\R^k) \text{ and } \abs{\rho}_{\Lin} \le 1\), then \(\abs{\rho \circ \xi}_{g^*_M\otimes g_k}\le \) \(\abs{\rho}_{\Lin} \abs{\xi}_{g^*_M\otimes g_n}\le \) \(\abs{\xi}_{g^*_M \otimes g_n}\).
On the other hand, since \(\dim(\im(\xi)) \le \min(m, n) \le k\), can choose \(\rho \in \Lin (\R^n,\R^k)\) such that \(\rho\) is an isometry on \(\im(\xi)\) and consequently \(\abs{\rho}_\mathcal{\Lin} \le 1\) and  \(\abs{\rho \circ \xi}_{g^*_M \otimes g_k} = \abs{\xi}_{g^*_M\otimes g_n}\).
\end{proof}

%PROOF PROPO CHARACTERIZATION
\begin{proof}[Proof of proposition~\ref{thm_equi}]
% 1 implique 2
Let us prove that \eqref{thm_equi1} implies \eqref{thm_equi2}. For every \(f\in C^1_c(N,\R^k)\), since \(u\) is colocally weakly differentiable, \(f\circ u\) is weakly differentiable. By proposition~\ref{thm_weak}, \(D(f\circ u) = Df \circ Du\) almost everywhere in \(M\), and so 
\[ 
  \abs{D(f\circ u)}_{\gFk} \le \abs{Df(u)}_{\Lin} \abs{Du}_{\gF} \le \abs{Df(u)}_{\Lin} \, w. 
\]

%2 implique 3
Since for every \(f\in C^1_c(N,\R^k)\) and for every \(y\in N\), \(\abs{Df(y)}_{\Lin} \le \abs{f}_{\Lip}\), the assertion \eqref{thm_equi2} implies directly \eqref{thm_equi3}. 

%3 IMPLIQUE 1
In order to prove that \eqref{thm_equi3} implies \eqref{thm_equi1} we first note that the map \(u\) is colocally weakly differentiable, and, by proposition~\ref{thm_weak}, has a unique colocal weak derivative \(Du\in \Hom(TM,TN)\).
Secondly, let \(U \subseteq N\), \(\varphi \in C^1(N,\R^n)\) and \(\varphi^* \in C^1(\R^n,N)\) be given by lemma~\ref{lemme_iso} for \(y\in N\) and \(\varepsilon >0\).
Since \(u = \varphi^* \circ \varphi \circ u\) on \(u^{-1} (U)\), by proposition~\ref{propo_equ}, \(Du = D(\varphi^* \circ \varphi \circ u)\) almost everywhere on \(u^{-1}(U)\).
By lemma~\ref{lemme_iso}, almost everywhere on \( u^{-1}(U)\), %by chain rule formula and according to definition of Frobenius norm,
\begin{equation}\label{equ_du}
\abs{Du}_{\gF} \le \abs{D\varphi^*(\varphi (u))}_{\Lin} \abs{D(\varphi \circ u)}_{g^*_M\otimes g_n} \le (1+\varepsilon) \abs{D(\varphi \circ u)}_{g^*_M\otimes g_n}.
\end{equation}
If \(\rho \in \Lin (\R^n,\R^k)\) and \(\abs{\rho}_{\Lin} \le 1\), by assumption, \(\abs{D(\rho \circ \varphi \circ u)}_{g^*_M\otimes g_k} \le \Lip(\rho \circ \varphi)\, w\) almost everywhere in \(M\). 
Since \(d_N\) is a geodesic distance, in view of lemma~\ref{lemme_iso}, \(\Lip(\varphi) = \sup_{z \in N} \abs{D\varphi(z)}_{\Lin} \le 1 + \varepsilon\). %BREZIS remark 7 p269
Hence, we have almost everywhere in \(M\)
\[
  \abs{\rho \circ D( \varphi \circ u)}_{g^*_M\otimes g_k}
  =\abs{D(\rho \circ \varphi \circ u)}_{g^*_M\otimes g_k}
  \le (1+\varepsilon) w.
\]
Since the set of nonexpansive linear maps is separable, we deduce from lemma~\ref{frob_norm} that
\[
  \abs{D( \varphi \circ u)}_{g^*_M \otimes g_n} \le (1 + \varepsilon) w
\]
almost everywhere in \(M\).
By inequality \eqref{equ_du}, we conclude that
\begin{equation*}\label{equ_du2}
\abs{Du}_{\gF} \le (1+\varepsilon)^2 w.
\end{equation*}
almost everywhere on \(u^{-1}(U)\).
We conclude by countable covering of \(N\).

%LAST ASSERTION
We now prove the last assertion. Let \(f \in \Lipk\).
By the approximation lemma~\ref{lemme_approx}, there exists a sequence \((f_\ell)_{\ell \in\mathbb{N}}\) in \(C^1_c(N,\R^k)\) such that \((f_\ell)_{\ell \in \mathbb{N}}\) converges uniformly to \(f\) over compact subsets  and \(\limsup_{l\to \infty} \abs{f_\ell}_{\Lip} \le \abs{f}_{\Lip}\). %La reference de Hebey \cite{hebey}*{theorem 2.7} ne couvre pas le cas \(p = \infty\).}
Since the sequence \((D(f_\ell \circ u))_{\ell \in \N}\) is bounded and uniformly integrable %equi / uniformly the same 
and since \((f_\ell \circ u)_{\ell \in \N}\) converges almost everywhere to \(f \circ u\), in view of the weak compactness criterion in \(L^1(M)\) \citelist{\cite{bogachev}*{corollary 4.7.19}\cite{brezis}*{theorem 4.30}}, the sequence \((D(f_\ell\circ u))_{\ell \in\mathbb{N}}\) converges weakly to \(D(f\circ u)\) in \(L^1_\mathrm{loc}(M)\) and \(f\circ u \in W^{1, 1}_\mathrm{loc} (M, \R^k)\). 
Moreover, for every \(v \in C^1_c (M, TM\otimes \R^k)\),
\begin{align*}
\Bigabs{\int_{M} \psh{D (f \circ u)}{v}} & = \lim_{\ell \to \infty} \Bigabs{\int_{M} \psh{D (f_\ell \circ u)}{v}} \\
& \le \liminf_{\ell \to \infty} \abs{f_\ell}_{\Lip} \int_{M} \abs{v}_{g_M^*\otimes g_k} w
\le \abs{f}_{\Lip} \int_{M} \abs{v}_{g_M^*\otimes g_k} w,
\end{align*}
and therefore
\(\abs{D (f \circ u)}_{\gFk} \le \abs{f}_{\Lip} w \) almost everywhere in \(M\).
\end{proof}

%COMPOSITION LIPSCHITZ MAP
Thanks to proposition~\ref{thm_equi}, we can consider the composition of a Lipschitz map from a manifold into an other with a map of homogeneous Sobolev space. 

\begin{propo}[Chain rule in Sobolev spaces]
\label{propositionCompositionLipschitz}
Let \((\Tilde{N}, g_{\Tilde{N}})\) be a Riemannian differentiable manifold.
Let \(u \in \Wp\) and let \(f \in \Lip(M, \Tilde{N})\). 
Then \(f\circ u \in \dot{W}^{1,p}(M, \Tilde{N})\) and
\[ 
  \abs{D(f\circ u)}_{g^*_M\otimes g_{\Tilde{N}}} \le \abs{f}_{\Lip} \abs{D u}_{\gF} \text{ almost everywhere in } M . 
\]
\end{propo}

This generalizes a well-known property (\citelist{\cite{brezis}*{proposition 9.5}\cite{lieb_loss}*{theorem 6.16}\cite{willem_en}*{proposition 6.1.13}});
the corresponding chain rule is more delicate \cite{ambrosio}. 

We also obtain a characterization of Sobolev spaces by embeddings.

\begin{propo}\label{equi_isometric}
Let \(\iota : N \to \Tilde{N}\) be an isometric embedding. For every \(u : M \to N\), \(u \in \dot{W}^{1, p} (M, N)\) if and only if \(\iota \circ u \in \dot{W}^{1, p} (M, \Tilde{N})\).
\end{propo}

In contrast with proposition~\ref{propoEmbeddingWeakDerivative}, the equivalence does not require \(\iota (N)\) to be closed.

\begin{proof}[Proof of propesition~\ref{equi_isometric}]
If \(u \in \dot{W}^{1, p} (M, N)\), then by the chain rule (proposition~\ref{propositionCompositionLipschitz}), \(\iota \circ u \in \dot{W}^{1, p} (M, \Tilde{N})\). Conversely, if \(\iota \circ u \in \dot{W}^{1, p} (M, \Tilde{N})\), then \(\iota \circ u\) is colocally weakly differentiable. By the weakly differentiable embedding property (proposition~\ref{propoEmbeddingWeakDerivative}), \(u\) is colocally weakly differentiable. By the chain rule (proposition~\ref{propoChainColocal}), \(D (\iota \circ u) = D \iota \circ Du\), and since the embedding \(\iota\) is isometric, \(\abs{D u}_{\gF} = \abs{D (\iota \circ u)}_{g^*_M\otimes g_{\Tilde{N}}}\).
\end{proof}

In particular our intrinsic definition is equivalent with the definition given by any embedding of \(N\) in a Euclidean space; such an embedding is always possible by the Nash embedding theorem~\citelist{\cite{nash1954}\cite{nash1956}}.

%%% CONVERGENCE AND DISTANCE  --------------------------------------------
\section{Sequences of colocally weakly differentiable maps}
\label{sectconvCWDM}

%WEAK CONVERGENCE SOBOLEV 
In this section we consider sequences of maps between differentiable manifolds without any fixed Riemannian structure. 
%
%FIRST COMPACTNESS in measure
We will first state a compactness theorem in measure that will rely on two boundedness assumptions.

%bilocally L^1 bounded 
\begin{de}
A sequence \((\upsilon_\ell)_{\ell \in \N}\) of bundle morphisms from \(TM\) to \(TN\)  that covers a sequence \((u_\ell)_{\ell \in \N}\) of maps from \(M\) to \(N\) is \emph{bilocally \(L^1\)--bounded} if for every \(x \in M\) and every \(y \in N\), there exist local charts \(\psi : V \subseteq M \to \R^m\) and \(\varphi : U \subseteq N \to \R^n\) such that \(x \in V\), \(y \in U\) and
\[
  \sup_{\ell \in \N} \int_{\psi (V \cap u_\ell^{-1} (U))} \abs{D\varphi \circ \upsilon_\ell \circ D(\psi^{-1})}_{g^*_m\otimes g_n} < \infty.
\]
\end{de}

Equivalently,  the sequence \((\upsilon_\ell)_{\ell \in \N}\) is bilocally \(L^1\)--bounded if there exists a positive measure \(\mu\) on \(M\) and a continuous norm \(\abs{\cdot}\) on the vector bundle \(T^* M \otimes TN\) such that
\[
 \sup_{\ell \in \N} \int_{M} \abs{\upsilon_\ell}\,d\mu < \infty.
\]

%LOCALLY COMPACT
\begin{de}
A sequence \((u_\ell)_{\ell \in \N}\) of maps from \(M\) to \(N\) is \emph{locally compact in measure} if for every \(x \in M\) there exists a local chart \(\psi : V \subseteq M \to \R^m\) such that \(x \in V\) and for every \(\varepsilon > 0\) there exists a compact set \(K \subseteq N\) such that for every \( \ell \in \N\),
\[
  \mathcal{L}^m \bigl(\psi(u_\ell^{-1}(N \setminus K) )\bigr) \le \varepsilon.
\]
\end{de}

If \(N\) is compact, then this condition is trivially satisfied. 
In general, a sequence \((u_\ell)_{\ell \in \N}\) of maps from \(M\) to \(N\) is locally compact in measure if and only if there exists a positive measure \(\mu\) on \(N\), \(y \in N\) and a continuous distance \(d\) on \(N\) such that 
\[
 \lim_{r \to \infty} \sup_{\ell \in \N} \mu \bigl(u_\ell^{-1}(N \setminus B^d_r (y))\bigr) = 0.
\]

%PROPO RELLICH
\begin{propo}[Compactness in measure]
\label{propo_rellich}
Let \((u_\ell)_{\ell \in \N}\) be a sequence of colocally weakly differentiable maps from \(M\) to \(N\).
If the sequence \((D u_\ell)_{\ell \in \N}\) is bilocally \(L^1\)--bounded  and the sequence \((u_\ell)_{\ell \in \N}\) is locally compact in measure, then there exists a measurable map \(u : M \to N\) and a subsequence \((u_{\ell_k})_{k \in \N}\) that converges to \(u\) almost everywhere in \(M\).
\end{propo}

\begin{proof}%[Proof of proposition~\ref{propo_rellich}]
Let \(\bigl((U_i, \varphi_i)\bigr)_{i \in I}\) be a family of extended local charts satisfying the conclusion of lemma~\ref{lemme_diffeo} such that \(\bigcup_{i \in I} U_i = N\) and \(I\) is countable.
Assume that \(\bigl(\eta_i\bigr)_{i \in I}\) is a \(C^1\)--partition of the unity subordinate to the covering \(\bigl(U_i\bigr)_{i \in I}\). We observe that in view of lemma~\ref{lemme_diffeo}, the set \(\Bar{U}_i\) is compact and hence \(\eta_i \in C^1_c (N,\R)\).

By assumption and by definition of the colocally weakly differentiability, for every \(i \in I\), the sequence \(((\varphi_i \circ u_\ell, \eta_i \circ u_\ell))_{\ell \in \N}\) is bounded in \(W^{1, 1}_{\mathrm{loc}} (M, \R^{n + 1})\). By the classical Rellich--Kondrashov compactness theorem \citelist{\cite{lieb_loss}*{theorem 8.9}\cite{brezis}*{theorem 9.16}\cite{willem_en}*{theorem 6.4.6}} and a diagonal argument, there exist a subsequence \((u_{\ell_k})_{k \in \N}\) and a negligible set \(E \subset M\) such that the sequence \((\varphi_i \circ u_{\ell_k} (x), \eta_i \circ u_{\ell_k} (x))_{k \in \N}\) converges in \(\R^{n + 1}\) for every \(x \in M \setminus E\) and every \(i \in I\).

We define the set
\[
F = \bigl\{x \in M \setminus E : \text{for every \(i \in I\), } \lim_{k \to \infty} \eta_i (u_{\ell_k} (x)) = 0 \bigr\}.
\]
For every compact set \(K \subseteq N\), we observe that 
\[
  F \subseteq \bigcup_{j \in \N} \bigcap_{k = j}^{\infty} u_{\ell_k}^{-1} (N \setminus K).
\]
Since the sequence \((u_\ell)_{\ell \in \N}\) is compact in measure, for every \(x \in M\), there is a local chart \(\psi : V \subseteq M \to \R^m\) such that \(x \in V\) and for every \(\varepsilon > 0\) there exists a compact set \(K \subseteq N\) such that for every \(k \in \N\),
\[
  \mathcal{L}^m \bigl(\psi (u_{\ell_k}^{-1}(N \setminus K) )\bigr) \le \varepsilon.
\]
Therefore,
\[
 \mathcal{L}^m (\psi (F \cap V)) \le \varepsilon.%
\]
Since \(\varepsilon > 0\) is arbitrary and \(M\) can be covered by countably many such charts, we conclude that the set \(F\) is negligible. 

We conclude by showing that \((u_{\ell_k})_{k \in \N}\) converges everywhere in \(M \setminus (E \cup F)\).
For every \(x \in M \setminus (E \cup F)\), by definition of the set \(F\), there exists \(i \in I\) such that \(\lim_{k \to \infty} \eta_i (u_{\ell_k} (x)) > 0 \). This implies, that for \(k \in \N\), large enough, \(u_{\ell_k} (x) \in U_i\). 
Since \(\varphi_i\) is a diffeomorphism on \(U_i\) and since the sequence \((\varphi_i (u_{\ell_k}(x)))_{k \in \N}\) converges, we define 
\(u (x) = {\varphi_i}_{\vert U_i}^{-1} \bigl(\lim_{k \to \infty} \varphi_i (u_{\ell_k}(x))\bigr)\) and we conclude that \((u_{\ell_k} (x))_{k \in \N}\) converges to \(u (x)\).
\end{proof}
%convergence in measure from M to N

A natural question is whether the limit of a sequence of colocally weakly differentiable maps is colocally weakly differentiable. We shall study this for sequences of maps converging in measure.

\begin{de} 
A sequence \((u_\ell)_{\ell \in \N}\) of maps from \(M\) to \(N\) \emph{converges locally in measure} to a map \(u : M \to N\) if for every \(x \in M\) there exists a local chart \(\psi : V \subseteq M \to \R^m\) such that \(x \in V\) and for every open set \(U \subseteq N\),
\[
  \lim_{\ell \to \infty} \mathcal{L}^m \bigl(\psi((u^{-1} (U) \cap V) \setminus u_\ell^{-1} (U))) = 0.
\]
\end{de}

If \(d\) is a continuous distance on \( N\) and if \(\mu\) is an absolutely continuous positive finite measure on \(M\), then the sequence \((u_\ell)_{\ell \in \N}\) converges to \(u\) locally in measure if and only if for every \(\varepsilon > 0\), 
\[
 \lim_{\ell \to \infty} \mu \left(\{x \in M : d (u_\ell (x), u (x)) > \varepsilon \}\right) = 0.
\]
This definition is consistent with the definition of convergence in measure of maps into  a metric space \cite{schwartz}*{d\'efinition 5.6.17}, which depends only on the topology of the space \cite{schwartz}*{th\'eor\`eme 5.6.21}.
Since the topologies of \(M\) and of \(N\) have a countable basis, it suffices to check the condition for a countable set of charts \(\psi: V \subseteq M \to \R^m\) and a countable collection of open sets \(U \subseteq N\). Hence, the convergence in measure for maps between manifolds induces a metrizable topology. Such a metric is given by 
\begin{equation}
\label{eqDistanceMeasure}
  \delta (u, v) = \int_M \frac{d (u, v)}{1+ d(u, v)} \,d\mu.
\end{equation}
Any Cauchy sequence with respect to \(\delta\) has a subsequence which is Cauchy almost everywhere 
\citelist{\cite{bogachev}*{exercise 4.7.60} \cite{schwartz}*{théorème 5.8.31}}.
Thus the space of measurable maps from \(M\) to \(N\) endowed with \(\delta\) is complete if and only if the space \((N, d)\) is complete.

The definition and the remarks applies directly to a sequence \((\upsilon_\ell)_{\ell \in \N}\) of bundle morphisms between \(TM\) and \(TN\) viewed as maps from \(M\) to \(T^*M \otimes TN\).

%bilocally uniformly integrable
\begin{de}
A sequence \((\upsilon_\ell)_{\ell \in \N}\) of bundle morphisms from \(TM\) to \(TN\) that covers a sequence \((u_\ell)_{\ell \in \N}\) of maps from \(M\) to \(N\) is \emph{bilocally uniformly integrable} if for every \(x \in M\) and every \(y \in N\), there exist local charts \(\psi : V \subseteq M \to \R^m\) and \(\varphi : U \subseteq N \to \R^n\) such that for every \(\varepsilon > 0\) there exists \(\delta > 0\) such that if \(W \subseteq \psi (V)\) and \(\mathcal{L}^m (W) \le \delta\) and \(\ell \in \N\)
\[
   \int_{W \cap \psi(u_\ell^{-1} (U))}  \abs{D\varphi \circ \upsilon_\ell \circ D(\psi^{-1})}_{g^*_m\otimes g_n} \le \varepsilon.
\]
\end{de}

Equivalently,  the sequence \((\upsilon_\ell)_{\ell \in \N}\) is bilocally uniformly integrable if there exists a positive absolutely continuous measure \(\mu\) on \(M\) and a continuous norm \(\abs{\cdot}\) on the vector bundle \(T^* M \otimes TN\) such that for every \(\varepsilon > 0\) there exists \(\delta > 0\) such that if \(V \subseteq M\), \(\mu (V) \le \delta\) and \(\ell \in \N\), then 
\[
   \int_{V}  \abs{\upsilon_\ell} \le \varepsilon.
\]

\begin{propo}[Closure property]
\label{propositionWeakClosure}
Let \((u_\ell)_{\ell \in \N}\) be a sequence of colocally weakly differentiable maps from \(M\) to \(N\).
If the sequence \((u_\ell)_{\ell \in \N}\) converges to \(u : M \to N\) locally in measure and if the sequence \((Du_\ell)_{\ell \in \N}\) is bilocally uniformly integrable, then the map \(u\) is colocally weakly differentiable, and for every map \(f \in C^1_c (N, \R)\), every local chart \(\psi : V \subseteq M \to \R^m\) and every test function \(v \in C^1_c (\psi(V), \R^m)\),
\[
  \lim_{\ell \to \infty} \int_{\psi (V)} \psh{ D (f \circ u_\ell \circ \psi^{-1})}{v} = \int_{\psi (V)} \psh{ D (f \circ u \circ \psi^{-1})}{v}.
\]
If moreover the sequence \((Du_\ell)_{\ell \in \N}\) converges to a bundle morphism \(\upsilon : TM \to TN\) locally in measure, then \(Du = \upsilon\).
\end{propo}

In particular, under the additional condition of bilocally uniform integrability of the sequence of colocal weak derivatives, the map given by proposition~\ref{propo_rellich} is colocally weakly differentiable.

\begin{proof}[Proof of proposition~\ref{propositionWeakClosure}]
Following classical argument \cite{willem_en}*{theorem 6.1.7}, given a local chart \(\psi : V \subseteq M \to \R^m\) and \(f \in C^1_c (N, \R)\), we define the linear functional \(F_{f, \psi}\) on \(C^1_c (\psi (V), \R^m)\) for every test function \(v \in C^1_c (\psi (V), \R^m)\) by
\[
 \psh{F_{f, \psi}}{v} = -\int_{\psi (V)} (f \circ u \circ \psi^{-1}) \operatorname{div} v.
\]
Let \(K \subset \psi (V)\) be compact. Since the sequence \((f \circ u_\ell \circ \psi^{-1})_{\ell \in \N}\) converges to \(f \circ u_\ell \circ \psi^{-1}\) in \(L^1  (K)\), if \(\supp v \subset K\),
\[
\begin{split}  
 \abs{\psh{F_{f, \psi}}{v}} &= \Bigabs{\int_{K} (f \circ u \circ \psi^{-1}) \operatorname{div} v} 
  = \lim_{\ell \to \infty} \Bigabs{\int_{K} (f \circ u_\ell \circ \psi^{-1}) \operatorname{div} v} \\
 & = \lim_{\ell \to \infty} \Bigabs{\int_{K} \psh{D (f \circ u_\ell \circ \psi^{-1})}{v}}
 \le \norm{v}_{L^\infty} \liminf_{\ell \to \infty} \int_{K} \abs{f}_{\Lip}\abs{D u_\ell}_{\gF}.
\end{split} 
\]
Therefore \(F_{f, \psi}\) is represented by a vector-valued Radon measure \(\mu_{f, \psi}\) on \(\psi (V)\):
\[
 \psh{F_{f, \psi}}{v} = \int_{\psi (V)} v \,  d\mu_{f, \psi}.
\]
We observe that for every open set \(W \subseteq \psi(V)\),
\[
 \abs{\mu_{f, \psi}} (W) \le \liminf_{\ell \to \infty} \int_{W} \abs{D (f \circ u_\ell \circ \psi^{-1})}_{g^*_m\otimes g_1}.
\]
By the uniform integrability assumption, we conclude that the measure \(\abs{\mu_{f, \psi}}\) is absolutely continuous with respect to the Lebesgue measure on every compact set \(K \subset \psi (V)\). The measure \(\mu_{f, \psi}\) can thus be represented by a vector-field in \(L^1_\mathrm{loc} (\psi (V))\). In particular the map \(f \circ u \circ \psi^{-1}\) is weakly differentiable and 
\[
  \lim_{\ell \to \infty} \int_{\psi (V)} \psh{ D (f \circ u_\ell \circ \psi^{-1})}{v} = \int_{\psi (V)} \psh{ D (f \circ u \circ \psi^{-1})}{v}.
\]

Finally, if \((Du_\ell)_{\ell \in \N}\) converges to a bundle morphism \(\upsilon : TM \to TN\) locally in measure , then the sequence \((Df \circ Du_\ell \circ D\psi^{-1})_{\ell \in \N}\) converges to \(D f \circ \upsilon \circ D \psi^{-1}\) in measure on \(\psi (V)\) . Since the sequence \((Df \circ Du_\ell \circ D\psi^{-1})_{\ell \in \N}\) is uniformly integrable on every compact set \(K \subset \psi (V)\), we conclude that \cite{bogachev}*{theorem 4.5.4}
\[
  \lim_{\ell \to \infty} \int_{\psi (V)} \psh{ D (f \circ u_\ell \circ \psi^{-1})}{v}
  = \int_{\psi (V)} \psh{D f \circ \upsilon \circ D \psi^{-1}}{v} .
\]
Since \(f \in C^1_c (N, \R)\), the chart \(\psi : V \to \R^m\) and \(v \in C^1_c (\psi(V),\R^m)\) are arbitrary, \(D u = \upsilon\).
\end{proof}

\section{Sequences of Sobolev maps between Riemannian manifolds}

\subsection{Weak convergence of Sobolev maps}
We now consider sequences of Sobolev maps. As a consequence of the results in the previous section, we have a Rellich-Kondrashov type compactness theorem.

%Rellich-Kondrashov type
\begin{propo}[Rellich--Kondrashov for Sobolev maps]
Let \((u_\ell)_{\ell \in \N}\) be a sequence of colocally weakly differentiable maps from \(M\) to \(N\), \(v : M \to N\) be measurable.
If \((N, d)\) is complete, if there exist \(p \in [1, \infty)\) such that 
\[
 \sup_{\ell \in \N} \int_{M} \abs{D u_\ell}_{\gF}^p  < \infty 
\]
and \(q \in [1, \infty)\) such that
\[
  \sup_{\ell \in \N} \int_{M} d (u_\ell, v)^q < \infty
\]
then there is a subsequence \((u_{\ell_k})_{k \in \N}\) that converges to \(u : M \to N\) almost everywhere in \(M\).
\end{propo}

\begin{proof}
Since the metric space \((N, d)\) is complete, for every \(y \in N\) and \(r \in \R\), the closed ball \(\Bar{B}^N_r (y)\) is compact. In particular the sequence \((u_\ell)_{\ell \in \N}\) is locally compact in measure.

On the other hand, the sequence \((D u_\ell)\) is bilocally \(L^1\)--bounded, and therefore by proposition~\ref{propo_rellich}, there exists a measurable map \(u : M \to N\) and a subsequence \((u_{\ell_k})_{k \in \N}\) that converges to \(u\) almost everywhere in \(M\). 
\end{proof}

As in classical theory, we have the following closure property, which will play an important role in the completeness of Sobolev spaces.

\begin{propo}[Weak closure property for Sobolev maps]
\label{prop_closure}
Let \(p  \ge 1\) and let \((u_\ell)_{\ell \in \N}\) be a sequence of colocally weakly differentiable maps from \(M\) to \(N\).
Assume that the sequence \((u_\ell)_{\ell \in \N}\) converges to \(u : M \to N\) locally in measure, that 
\[
 \liminf_{\ell \to \infty} \int_{M} \abs{D u_\ell}_{\gF}^p < \infty,
\]
and, if \(p = 1\), that the sequence \((D u_\ell)_{\ell \in \N}\) is locally uniformly integrable.
Then \(u \in \dot{W}^{1, p} (M, N)\),
\[
 \int_{M} \abs{D u}_{\gF}^p \le \liminf_{\ell \to \infty} \int_{M} \abs{D u_\ell}_{\gF}^p,
\]
and for every \(f \in C^1_c (N, \R)\) and every section \(v : M \to  TM\), such that \(\abs{v}_{g_M} \in L^{p/(p - 1)} (M)\),
\[
  \lim_{\ell \to \infty} \int_{M} \psh{ D (f \circ u_\ell)}{v} = \int_{M} \psh{ D (f \circ u)}{v}.
\]
If moreover the sequence \((Du_\ell)_{\ell \in \N}\) converges to a bundle morphism \(\upsilon : TM \to TN\) locally in measure, then \(Du = \upsilon\).
\end{propo}

The Euclidean counterpart of this property is well-known \cite{willem_en}*{theorem 6.1.7}.
The uniform integrability assumption is essential for \(p = 1\): otherwise the closure property fails already for classical Sobolev maps between Euclidean spaces.

\begin{proof}[Proof of proposition~\ref{prop_closure}]%
By the boundedness and bilocally uniform integrability assumptions, the sequence of bundle morphisms \((Du_\ell)_{\ell \in \N}\) is bilocally uniformly integrable. Proposition~\ref{propositionWeakClosure} applies and it remains to prove that \(\abs{Du}_{\gF} \in L^p (M)\).
By the boundedness and bilocally uniform integrability assumptions, up to a subsequence, the sequence of functions \((\abs{D u_\ell}_{\gF})_{\ell \in \N}\) converges weakly to some \(w \in L^p (M)\).
For every \(f \in C^1_c (N, \R^k)\) with \(k=\min (m, n)\) and \(\ell \in \N\), we have 
\[
  \abs{D(f\circ u_\ell)}_{\gFk} \le \abs{f}_{\Lip}\abs{Du_\ell}_{\gF}. 
\]
Hence, for every \(v \in C^1_c(M, TM \otimes \R^k)\),
\[
\begin{split}
 -\int_{M} \psh{f\circ u_\ell}{\operatorname{div} v} = \int_{M} \psh{D(f\circ u_\ell)}{v} &\le  \int_{M} \abs{D(f\circ u_\ell)}_{\gFk} \abs{v}_{g^*_M \otimes g_k}\\ & \le \abs{f}_{\Lip} \int_{M}\abs{Du_\ell}_{\gF} \abs{v}_{g^*_M \otimes g_k}
\end{split}
\]
and thus 
\[
 - \int_{M} \psh{f\circ u}{\operatorname{div} v} \le \abs{f}_{\Lip} \int_{M} \abs{v}_{g^*_M \otimes g_k} w.
\]
Since the map \(u\) is colocally weakly differentiable, we deduce that 
\[
 \int_{M} \psh{D (f\circ u)}{v} \le \abs{f}_{\Lip} \int_{M} \abs{v}_{g^*_M \otimes g_k} w.
\]
Since \(v \in C^1_c(M, TM \otimes \R^k)\) is arbitrary we conclude that 
\[
 \abs{D (f \circ u)}_{\gFk} \le \abs{f}_{\Lip} w
\]
almost everywhere in \(M\). By the characterization of the norm of the derivative of proposition~\ref{thm_equi}, \(\abs{D u}_{\gF} \le w\) almost everywhere in \(M\) and thus by lower semicontinuity  of the norm under weak convergence \citelist{\cite{brezis}*{proposition 3.5} \cite{lieb_loss}*{theorem 2.11}\cite{willem_en}*{theorem 5.4.6}}
\[
  \int_{M} \abs{Du}_{\gF}^p \le \int_{M} w^p \le \liminf_{\ell \to \infty} \int_{M} \abs{Du_\ell}_{\gF}^p.\qedhere
\]
\end{proof}

%CONVERGENCE IN SOBOLEV 
\subsection{Strong convergence in Sobolev spaces}

We define now a notion of convergence in homogeneous Sobolev spaces \(\dot{W}^{1, p} (M, N)\).

%CONVERGENCE IN SOBOLEV
\begin{de}
\label{def_convergence_sobolev}
Let \(p \in [1,\infty)\).
The sequence \((u_\ell)_{\ell \in \N}\) in \(\dot{W}^{1, p} (M, N)\) \emph{converges strongly} to \(u\in \dot{W}^{1, p} (M, N)\) \emph{in \(\dot{W}^{1, p} (M, N)\)} if \((Du_\ell)_{\ell \in \N}\) converges to \(Du\) locally in measure and \((\abs{Du_\ell}_{\gF})_{\ell \in \N}\) converges to \(\abs{Du}_{\gF}\) in \(L^p (M)\).
\end{de}

The convergence of definition~\ref{def_convergence_sobolev} is induced by the distance
\[
  \dot{\delta}_{1,p} (u, v) = \delta (Du, Dv) + \norm{\abs{Du}_{\gF} - \abs{Dv}_{\gF}}_{L^p (M)},
\]
where \(\delta\) is a distance of the form \eqref{eqDistanceMeasure}.
In fact, a sequence \((u_\ell)_{\ell \in \N}\) converges strongly to \(u\) in \(\dot{W}^{1, p} (M, N)\) if and only if \((Du_\ell)_{\ell \in \N}\) converges to \(Du\) locally in measure and 
\[\lim_{\ell \to \infty} \int_M \abs{Du_\ell}_{\gF}^p = \int_M \abs{Du}_{\gF}^p.\]
This follows from the Euclidean counterpart \citelist{\cite{bogachev}*{proposition 4.7.30}\cite{willem_en}*{theorem 4.2.6}\cite{brezis}*{exercise 4.17.3}}. 

\begin{propo}
If for almost every \(x \in M\), \((\{x\} \times N, d)\) is complete and if the projection \(\pi_{M \times N} : T^*M \otimes TN \to M \times N \subset T^*M \otimes TN\) is nonexpansive with respect to \(d\),
then the metric space \((\dot{W}^{1, p} (M, N), \dot{\delta}_{1, p})\) is complete.
\end{propo}
\begin{proof}
Let \((u_\ell)_{\ell \in \N}\) be a Cauchy sequence for the distance \(\dot{\delta}_{1,p}\).
By the completeness of \(L^p (M)\), there exists a map \(w \in L^p (M)\) such that \((\abs{D u_\ell}_{\gF})_{\ell \in \N}\) converges to  \(w\) in \(L^p (M)\).
By the properties of the distance defined in \eqref{eqDistanceMeasure}, there is a subsequence \((Du_{\ell_k})_{k \in \N}\) which is Cauchy almost everywhere with respect to the distance \(d\) \citelist{\cite{bogachev}*{exercise 4.7.60} \cite{schwartz}*{théorème 5.8.31}} and \((\abs{D u_{\ell_k}}_{\gF})_{k \in \N}\) converges to  \(w\) almost everywhere on \(M\). 
By the nonexpansiveness of the fiber projection \(\pi_{M \times N}\), for almost every \(x \in M\), \(((x, u_{\ell_k} (x)))_{k \in \N}\) is a Cauchy sequence with respect to the distance \(d\). 
Since for almost every \(x \in M\), the space \((\{x\} \times N, d)\) is complete, the sequence \((u_{\ell_k} )_{k \in \N}\) converges almost everywhere to a measurable map \(u : M \to N\).
Since \(d\) is continuous, it is complete on every compact subset of \(T^*M \otimes TN\) \cite{munkres}*{theorem 45.1}. Therefore there exists a measurable bundle morphism \(\upsilon : TM \to TN\) such that \((Du_{\ell_k})_{k \in \N}\) converges almost everywhere to \(\upsilon\) and \(\abs{\upsilon}_{\gF} = w\).
By the weak closure property (proposition~\ref{prop_closure}), \(\upsilon = D u\). Therefore, \((Du_{\ell_k})_{k \in \N}\) converges almost everywhere to \(D u\) and thus with respect to \(\delta\). Since the sequence \((Du_{\ell})_{\ell \in \N}\) is Cauchy with respect to the distance \(\delta\), the whole sequence \((Du_{\ell})_{\ell \in \N}\) converges to \(Du\) with respect to the distance \(\delta\).
\end{proof}

Given an isometric embedding \(\iota : N \to \Tilde{N}\), we can also consider the distance 
\(\dot{\delta}_{1, p}^\iota (u, v) = \dot{\delta}_{1, p} (\iota \circ u, \iota \circ v)\). This distance gives the same convergence.

\begin{propo}
If \(\iota : N \to \Tilde{N}\) is an isometric embedding, then the sequence \((u_\ell)_{\ell \in \N}\) converges strongly to \(u \in W^{1, p} (M, N)\) in \( \dot{W}^{1, p}(M, N)\) if and only if 
the sequence \((\iota \circ u_\ell)_{\ell \in \N} \) converges strongly to \(\iota \circ u \in \dot{W}^{1,p}(M,\Tilde{N})\) in \(\dot{W}^{1,p}(M,\Tilde{N})\).
\end{propo}
\begin{proof}
Since \(\iota\) is an embedding, \((D (\iota \circ u_\ell))_{\ell \in \N}\) converges to \(D(\iota \circ u)\) in measure if and only if \((D u_\ell)_{\ell \in \N}\) converges to \(D u\) locally in measure.
As \(\iota\) is isometric, for every \(\ell \in \N\),
\[
  \abs{D (\iota \circ u_\ell)}_{g^*_M\otimes g_{\Tilde{N}}} - \abs{D (\iota \circ u)}_{g^*_M\otimes g_{\Tilde{N}}} = \abs{D u_\ell}_{\gF} - \abs{D u}_{\gF};
\]
the conclusion follows from the definition of convergence.
\end{proof}

The distance \(\dot{\delta}_{1, p}^\iota\) gives thus the same topology as \(\dot{\delta}_{1, p}\). However the completeness of \(\dot{W}^{1, p} (M, N)\) will then depend on the completeness of \(\{x\} \times \iota (N)\); a necessary condition is that \(\iota (N)\) should be closed.
When \(N\) is complete but not compact, the original Nash embedding theorem will give \(\iota(N)\) which is not closed \citelist{\cite{nash1954}\cite{nash1956}}; it is however always possible when \(N\) is complete to have a Nash embedding theorem with \(\iota (N)\) closed \cite{muller}.

\medskip

%COMPARE CHIRON DISTANCE
We would also like to compare these notion with the metric of Chiron \cite{chiron}*{\S 1.6}\footnote{As the modulus of the derivative has several definitions, we have in fact a family of distances out of which we take the one that uses our notion of modulus of the derivative. Instead of introducing the notion of Lebesgue spaces into metric spaces, we consider convergence in measure for maps from \(M\) to \(N\).} for \(u,v\in \dot{W}^{1,p}(M, N)\): 
\[ 
  \dot{\delta}_{1,p}^C (u,v) = \delta (u, v) + \Bigl(\int_{M} \bigabs{\abs{D u}_{\gF} - \abs{D v}_{\gF}}^p\Bigr)^\frac{1}{p}. 
\]
The Sobolev space is \emph{not complete} under this distance when \(N = \R\) \cite{chiron}*{lemma 2}.
In order to study the topological equivalence of these metrics we give a criterion of convergence in measure of the derivative.

\begin{propo}[Criterion for convergence in measure]
\label{propositionCriterionStrongConvergence}
Let \((u_\ell)_{\ell \in \N}\) be a sequence of colocally weakly differentiable maps from \(M\) to \(N\).
If the sequence \((u_\ell)_{\ell \in \N}\) converges locally in measure to a colocally weakly differentiable map \(u \colon M \to N\) and if the sequence \((\abs{D u_\ell}_{\gF})_{\ell \in \N}\) converges to \(\abs{Du}_{\gF}\) in \(L^1_\mathrm{loc} (M)\), then the sequence \((Du_\ell)_{\ell \in \N}\) converges to \(D u\) locally in measure.
\end{propo}

As an immediate consequence we have the topological equivalence between \(\dot{\delta}_{1,p}\) and \(\dot{\delta}_{1,p}^C\).

\begin{propo}
\label{corChironEquiv}
Let \(p \in [1, \infty)\).
Let \((u_\ell)_{\ell \in \N}\) be a sequence of maps in \(\dot{W}^{1, p} (M, N)\) and \(u \in \dot{W}^{1, p} (M, N)\). The sequence \((u_\ell)_{\ell \in \N}\) converges strongly to \(u\) in \( \dot{W}^{1,p}(M,N) \) if and only if \(\lim_{\ell \to \infty} \dot{\delta}^C_{1, p} (u_\ell, u) = 0\).
\end{propo}

This proposition is due to Chiron when \(N = \R\) and \(p > 1\). The proof of proposition~\ref{propositionCriterionStrongConvergence} relies on the Balder-Visintin criterion of strong convergence.

\begin{propo}[Balder--Visintin criterion of strong convergence \citelist{\cite{balder}*{theorem~1}\cite{visintin}*{corollary 2}}]
\label{propo_baldervisitin}
Let \((f_\ell)_{\ell \in \N}\) be a sequence in \(L^1 (U, \R^k)\). 
If the sequence \((f_\ell)_{\ell \in \N}\) converges weakly to \(f \in L^1 (U,\R^k)\) and for almost every \(x \in U\), the point \(f (x)\) is an extreme point of
\[
 \bigcap_{j \in \N} \overline{\operatorname{co} \{f_i (x) \colon i \ge j \}},
\]
then the sequence \((f_\ell)_{\ell \in \N}\) converges to \( f\) in \( L^1(U,\R^k)\).
\end{propo}

For a set \(A \subseteq \R^k\), the set \(\operatorname{co} A\) denotes the \emph{convex hull} of \(A\), that is, the set of convex combinations of elements of \(A\). A point \(c\) is an \emph{extreme point} of a convex set \(C\) if \(C \setminus \{c\}\) is convex.

\begin{proof}[Proof of proposition~\ref{propositionCriterionStrongConvergence}]
Let \(U\subseteq N \) and \(\varphi \in C^1 (N,\R^n)\) be the extended local chart given by lemma~\ref{lemme_diffeo} and \(K \subseteq M\) be compact. 
In view of proposition~\ref{propositionWeakClosure}, and the weak compactness criterion in \(L^1_{\mathrm{loc}} (K)\), the sequence \((D (\varphi \circ u_\ell))_{\ell \in \N}\) converges weakly to \(D (\varphi \circ u)\) in \(L^1 (K)\). 

By taking a subsequence, we can assume that the sequences \((u_\ell)_{\ell \in \N}\) and \((\abs{D u_\ell})_{\ell \in \N}\) converge almost everywhere in \(M\).
In order apply the Balder--Visintin criterion of strong convergence, we note that for every \(x \in M\),
\[
 D (\varphi \circ u_\ell) (x) \in D \varphi (u_\ell (x)) (\Bar{B}_{\abs{D u_\ell (x)}_{\gF}})
\]
and so
\[ 
\overline{\operatorname{co} \{D(\varphi \circ u_k) (x) \colon k \ge \ell \}} \subseteq \overline{\operatorname{co} \{ D\varphi(u_k(x))(\Bar{B}_{\abs{D u_k (x)}_{\gF}}) \colon k \ge \ell \}}.
\]
Hence, since for almost every \(x \in K\), \((D \varphi (u_\ell (x)))_{\ell \in \N}\) converges to \(D \varphi (u (x))\) and \( (\abs{Du_\ell(x)}_{\gF})_{\ell \in \N} \) converges to \( \abs{Du(x)}_{\gF} \), we have 
\[
 \bigcap_{\ell \in \N} \overline{\operatorname{co} \{D(\varphi \circ u_k) (x) \colon k \ge \ell \} } \subseteq D \varphi (u (x))(\Bar{B}_{\abs{D u (x)}_{\gF}}).
\]
We finally observe that for every \(x \in u^{-1} (U)\), \( \abs{Du(x)}_{\gF}\) is an extremal point of \(\Bar{B}_{\abs{D u (x)}_{\gF}}\) and \(D \varphi (u (x))\) is invertible, therefore \(D(\varphi \circ u) (x)\) is an extremal point of \(D \varphi (u (x))(\Bar{B}_{\abs{D u (x)}_{\gF}})\) and hence of \(\bigcap_{\ell \in \N} \overline{\operatorname{co} \{D(\varphi \circ u_k) (x) \colon k \ge \ell \} }\). Hence by the Balder--Visintin criterion of strong convergence (proposition~\ref{propo_baldervisitin}), the sequence \((D (\varphi \circ u_\ell))_{\ell \in \N}\) converges to \(D (\varphi \circ u)\) in \(L^1 (K \cap u^{-1} (U))\) and thus in measure on \(K \cap u^{-1} (U)\). By covering \(M\) and \(N\) by countably many such compact sets \(K\) and extended local charts \((\varphi, U)\), we obtain the conclusion.
\end{proof}

The reader will observe that our argument relies on the structure of the norm that endows \(T^*M \otimes TN\). More precisely our proof requires the norm on \(T^*M \otimes TN\) to be uniformly convex; this is not the case when \(\min (\dim (M),\dim (N))\ge 2\) for the operator norm.

%COROLLAIRE EQUIVALENT WITH p>1
\begin{propo}
\label{propo_chiron2}
Let \(p \in (1, \infty)\).
Let \((u_\ell)_{\ell \in \N}\) be a sequence of maps in \(\dot{W}^{1, p} (M, N)\) and \(u \in \dot{W}^{1, p} (M, N)\). If \( (u_\ell)_{\ell \in \N}\) converges to \(u\) locally in measure and 
\[
 \lim_{\ell \to \infty} \int_{M} \abs{Du_\ell}_{\gF}^p
 = \int_{M} \abs{Du}_{\gF}^p,   
\]
then \( (\abs{Du_\ell}_{\gF})_{\ell \in \N} \) converges to \(\abs{Du}_{\gF}\) in \( L^p(M)\).
\end{propo}

In particular the two metrics introduced by Chiron have the same convergent sequences \cite{chiron}*{lemma 2} for \(p > 1\).
By proposition~\ref{corChironEquiv}, this notion of convergence is also equivalent with convergence in \( \dot{W}^{1,p}(M,N)\). When \( p = 1 \), the equivalence does not hold already in the Euclidean case \cite{chiron}*{lemma 2}. The proof relies on the proposition~\ref{thm_equi}. 

\begin{proof}[Proof of proposition~\ref{propo_chiron2}]
Since \( (\abs{Du_\ell}_{\gF})_{\ell \in \N}\) is bounded in \( L^p(M) \) and \(p \in (1, \infty)\), by taking a subsequence, we can assume that \((\abs{Du_\ell}_{\gF})_{\ell \in \N}\) converges weakly to some \( w\) in \( L^p(M)\). 
Since \( (u_\ell)_{\ell \in \N}\) converges to \(u\) locally in measure, for every \( f \in C^1_c(M,\R^k)\), \( \abs{D(f\circ u)}_{\gFk} \le \abs{f}_{\Lip} \, w \) almost everywhere in \( M\). Hence, by proposition~\ref{thm_equi}, \( \abs{Du}_{\gF} \le w \) almost everywhere in \(M\). On the other hand, by lower semicontinuity  of the norm under weak convergence \citelist{\cite{brezis}*{proposition 3.5} \cite{lieb_loss}*{theorem 2.11}\cite{willem_en}*{theorem 5.4.6}} and by our assumption,
\[
  \int_{M} w^p \le \liminf_{\ell \to \infty} \int_{M} \abs{Du_\ell}_{\gF}^p
  = \int_{M} \abs{D u}_{\gF}^p
\]
and so \(w = \abs{Du}_{\gF}\) almost everywhere in \(M\). 
The sequence \((\abs{Du_\ell}_{\gF})_{\ell \in \N}\) converges thus weakly to \(\abs{D u}_{\gF}\) in \(L^p (M)\). Since \(p \in (1, \infty)\), we conclude that \((\abs{Du_\ell}_{\gF})_{\ell \in \N}\) converges to \(\abs{D u}_{\gF}\) in \(L^p (M)\) \citelist{\cite{bogachev}*{corollary 4.7.16}\cite{willem_en}*{exercise 5.3}\cite{brezis}*{exercise 4.19}}.
\end{proof}

We end this section by studying the continuity of Sobolev-type embeddings: Sobolev spaces of maps between manifolds essential inherit continuous Sobolev embeddings of Sobolev spaces of scalar functions.

\begin{propo}[Continuity of Sobolev type embeddings]
If \(u \in \dot{W}^{1, p} (M, N)\) and if \(\dot{W}^{1, p} (M) \cap L^q (M)\) is embedded in \(L^r (M)\), then for every \(\varepsilon > 0\), there exists \(\delta > 0\) such that if \(v \in \dot{W}^{1, p} (M; N)\) and \(\dot{\delta}_{1, p} (u, v) + \norm{d  (u, v)}_{L^q(M)} \le \delta\), then \(\norm{d (u, v)}_{L^r(M)} \le \varepsilon\).
\end{propo}

Since we are in a nonlinear setting, the continuity is not a consequence of the boundedness of the embedding operator; it will follow instead of the continuity in Sobolev spaces of the distance function.

\begin{propo}\label{propo_continuity_distance}
Let \(p \in [1,\infty) \).
If the sequence \((u_\ell)_{\ell \in \N}\) in \(\dot{W}^{1, p} (M, N)\) converges strongly to \(u\in \dot{W}^{1, p} (M, N)\) in \(\dot{W}^{1, p} (M, N)\), then the sequence \((d(u_\ell, u))_{\ell \in \N}\) converges strongly to \(0\) in \(\dot{W}^{1, p} (M)\).
\end{propo}

\begin{propo}
Let \((u_\ell)_{\ell \in \N}\) be a sequence of colocally weakly differentiable maps from \(M \) to \(N\). 
If the sequence \((D u_\ell)_{\ell \in \N}\) converges to \(Du\) locally in measure, then \((D (d (u_\ell, u)))_{\ell \in \N} \) converges to \(0\) in measure.
\end{propo}
\begin{proof}
This follows from the fact that for the geodesic distance \(d\), for every \(\xi \in TN\), 
\[
 \lim_{\substack{\zeta \to \xi\\ \pi_N (\xi) \ne \pi_N (\xi)}} (D d)(\xi, \zeta) = 0,
\] 
and that \(Du_n = Du\) almost everywhere on the set where \(u_n = u\) (proposition~\ref{propo_equ}).
\end{proof}

\begin{proof}[Proof of proposition \ref{propo_continuity_distance}]
The property of convergence in measure is a consequence of the previous proposition. Since for every \(\ell \in \N\), almost everywhere in \(M\) 
\[ \abs{D(d(u_l,u))} \le \abs{Du_\ell}_{\gF} + \abs{Du}_{\gF},\]
by Lebesgue's dominated convergence theorem, the sequence \( (\abs{Du_\ell}_{\gF})_{\ell \in \N}\) converges to  \( \abs{Du}_{\gF}\) in \(L^p(M)\).

\end{proof}

%CONCORDANT DISTANCES
\subsection{Concordant distances and metrics}
In order to study more natural distances of the form
\[
  \dW(u, v) = \left(\int_M \dTN(Du,Dv)^p\right)^\frac{1}{p},
\]
we introduce and study concordant distances and metrics on fiber bundles \cite{blw}.

%concordant distance
\begin{de}
A continuous distance \(\dE\) on a normed topological vector bundle \((E, \pi_M, M)\) is \emph{concordant} with the norm \(\abs{\cdot}_E\)
if
\begin{enumerate}[(a)]
  \item (nonexpansiveness of the fiber projection) for every \(e_1,e_2 \in E\),
    \begin{equation*}
      \dE (\pi_M (e_1), \pi_M (e_2)) \le \dE(e_1, e_2);
    \end{equation*} 
  \item (equivalence with the norm) there exists \(\kappa > 0\) such that for every \(e \in E\),
    \begin{equation*}
      \kappa^{-1} \, \abs{e}_{E} \le \dE( e, \pi_M (e) ) \le \kappa \abs{e}_{E}.
    \end{equation*} 
\end{enumerate}
\end{de}

Here and in the sequel we identify \(M\) with the zero fiber subbundle \(M \times \{0\} \subset E\).

\begin{propo}[Stability of subbundles]
Let \((E, \pi_M, M)\) and \((\Tilde{E}, \pi_{\Tilde{M}}, \Tilde{M})\) be vector bundles 
and \(\Bar{\iota} \colon E \to \Tilde{E}\) be an injective bundle morphism that covers an embedding \(\iota \colon M \to \Tilde{M}\).
If \(d_{\Tilde{E}}\) is concordant with \(\abs{\cdot}_{\Tilde{E}}\), then the distance \(d_E\) defined for \(e_1, e_2 \in E\) by 
\[
  \dE (e_1, e_2) = d_{\Tilde{E}} (\Bar{\iota}(e_1), \Bar{\iota} (e_2))
\]
is concordant with \(\abs{\cdot}_{E}\) defined for \(e \in E\) by \(\abs{e}_{E} = \abs{\Bar{\iota} (e)}_{\Tilde{E}}\).
\end{propo}

If \(E\) is a differentiable manifold, we study Riemannian metrics \(G_E\) on \(TE\) that give rise to a distance concordant with \(\abs{\cdot}_{g_E}\).
We denote by \(G_E\) and \(g_E\) the quadratic forms associated to the corresponding metrics. We recall that the \emph{vertical lift} is defined for each \(\nu \in \pi^{-1}_M(\{\pi_M(e)\})\) by
\begin{equation*}
 \Vertlift_e (\nu) = \frac{d}{dt} (e + t \, \nu)_{\arrowvert_{t=0}} \in T_e E.
\end{equation*}

%stronlgy concordant metric
\begin{de}
A metric \(G_E\) is \emph{strongly concordant} with the Euclidean structure \(g_E\) if 
\begin{enumerate}[(a)]
  \item for every \(\nu \in TE\),
    \[
      G_E (D \pi_M (\nu)) \le G_E (\nu); % in the sense (D \pi_M( \nu), 0) \in TM \times TE 
    \]
  \item there exists \(\kappa > 0\) such that for every \(\nu \in TE\),
    \[
      (\kappa^{-1} D g_E (\nu))^2 \le 4 g_E (\pi_{E} (\nu)) G_E (\nu),
    \]
  and for every \(e \in E\),
    \[
      G_E (\Vertlift_e (e)) \le \kappa^2 g_E (e).      
    \]
\end{enumerate}
\end{de}

We will show that strongly concordant metrics induce concordant distances. In fact, these distances have the stronger property of strong concordance.

%strongly concordant
\begin{de}
A continuous distance \(\dE\) on a normed topological vector bundle \((E, \pi_M, M)\) is \emph{strongly concordant} with the norm \(\abs{\cdot}_E\)
if
\begin{enumerate}[(a)]
  \item (nonexpansiveness of the fiber projection) for every \(e_1,e_2 \in E\),
    \begin{equation*}
      \dM(\pi_M (e_1), \pi_M (e_2)) \le \dE(e_1, e_2);
    \end{equation*} 
  \item (comparability with the norm) there exists \(\kappa > 0\) such that for every \(s, t \in \R\) and \(e \in E\),
  \[
    \dE (t e, s e) \le \kappa \abs{t e - s e}_E,
  \]
  and for every \(e_1, e_2 \in E\),
  \[
    \abs{\abs{e_1}_E - \abs{e_2}_E} \le \kappa \dE (e_1, e_2).
  \]

\end{enumerate}
\end{de}

\begin{propo}
The metric \(G_E\) is strongly concordant with \(g_E\) if and only if the geodesic distance \(d_E\) induced by the metric \(G_E\) is strongly concordant with \(\abs{\cdot}_{g_E}\). 
\end{propo}

\begin{proof}
Let \(e_1, e_2\in E\) and let \(\gamma \in C^1 ([0,1], E)\) be a path
such that \(\gamma (0) = e_1\) and \(\gamma (1) = e_2\). 
Since \(G_E\) is strongly concordant with \(g_E\),
\[
  \int_0^1 \sqrt{G_E (D \pi_M (\gamma' (\tau)))}\,d\tau
  \le \int_0^1 \sqrt{G_E (\gamma' (\tau))}\,d\tau,
\]
and the divergence property of the fibers follows.

Let \(e \in E\) and let \(\gamma \in C^1 ([0,1], E)\) be defined for every \(\tau \in [0, 1]\) by \(\gamma (\tau) = (\tau t + (1-\tau)s) e\). By definition of the vertical lift, for every \(\tau \in [0, 1]\), we have \(\gamma' (\tau) = \Vertlift_{\gamma (\tau)} ((t - s)e)= \tau \Vertlift_{(t-s)e}((t-s)e)\). 
Consequently,
\begin{align*}
  d_E(s e,t e) & \le \int_0^1 \sqrt{G_E(\Vertlift_{\gamma(\tau)} (e))} \, d\tau \\
  & \le \int_0^1 \sqrt{G_E(\Vertlift_{(t-s)e} ((t-s)e)} \, d \tau \le \kappa \sqrt{g_E((t-s)e)}. 
\end{align*}

Let \(\gamma \in C^1 ([0,1], E)\) be a path such that \(\gamma (0) = e_1\) and \(\gamma (1) = e_2\).
By assumption
\begin{align*}
\sqrt{g_E(e_1)} - \sqrt{g_E(e_2)} & = \int_0^1 \frac{d}{d\tau} \sqrt{g_E(\gamma (\tau))} \, d\tau = \int_0^1 \frac{D g_E(\gamma'(\tau))}{2 \sqrt{g_E(\gamma (\tau))}} \, d\tau \\
& \le \kappa \int_0^1 \sqrt{G_E(\gamma'(\tau))}\, d\tau, 
\end{align*}
so \(\abs{\abs{e_1}_{g_E} - \abs{e_2}_{g_E}} \le \kappa d_E(e_1, e_2)\).
\end{proof}

We shall now show that the Sasaki \cite{sasaki} and the Cheeger--Gromoll metrics \cite{cheeger_gromoll}, which are two classical constructions of natural metrics on bundles \citelist{\cite{gud_kappos}\cite{kappos}}, are strongly concordant.

In order to define these metrics, we endow \(M\) with a Riemannian metric \(g_M\) and we endow \(E\) with a metric connection \(K_E \colon TE \to E\). 
We recall that \(K_E\) is a \emph{connection} if \(K_E\) is a bundle morphism from \((TE, \pi_E, E)\) to \((E, \pi_M, M)\) that covers \(\pi_M\) and a bundle morphism from \((TE, D \pi_M, TM)\) to \((E, \pi_M, M)\) that covers \(\pi_M \colon TM \to M\), and for every \(e \in E\), \(K_E \circ \Vertlift_e = \id\) on \(\pi^{-1}_M(\{\pi_M (e)\}\)) \cite{wendl}*{definition 3.9}. 
The connection \(K_E\) is \emph{metric} if for every \(\nu \in TE\),
\[
  D g_E (\nu) = 2 g_E (K_E (\nu),\pi_E (\nu)).
\]

The \emph{Sasaki metric} \(G_E^S\) \cite{sasaki} (see also \cite{docarmo}*{chapter 3, exercise 2}) is defined for every \(\nu \in TE\) by
\[
  G_E^S (\nu) = g_M (D \pi_M (\nu)) + g_E (K_E (\nu)).
\]

%sasaki metric
\begin{propo}\label{propo_sasaki}
The Sasaki metric \(G_E^S\) is strongly concordant with \(g_E\).
\end{propo}

\begin{proof}
Since \(K_E (D \pi_M (\nu))= 0\), we have
\[
  G_E^S (D \pi_M (\nu)) 
  = g_M (D \pi_M (\nu)) \le G_E^S (\nu).
\]
Next, since the connection \(K_E\) is metric, we have for every \(\nu \in TE\),
\[ 
  (Dg_E (\nu))^2
  = 4 (g_E (K_E (\nu),\pi_M(\nu)))^2
  \le 4  g_E (K_E (\nu)) g_E (\pi_E(\nu))
  \le 4 G_E^S (\nu) g_E (\pi_E(\nu)).
\]
Finally, since \(D \pi_M (\Vertlift_e(e)) = 0\) and \(K_E (\Vertlift_e (e)) = e\), we conclude that 
\[
  G_E^S (\Vertlift_e (e)) = g_E (e). \qedhere
\]
\end{proof}

If \(d_E^S\) denotes the Sasaki geodesic distance, we have the following characterization \cite{sasaki}*{\S 2}:

\begin{propo}\label{lemme_sas_equi}
Let \(e_1, e_2 \in E\). Then
\begin{multline*}
d_E^S (e_1,e_2)^2 = \inf \Bigl\{ g_E (e_1 - P^{\gamma}(e_2)) + \int_0^1 g_M (\gamma' (\tau))\,d\tau \, \colon \,\\ \gamma \in C^1 ([0,1], M) ;  
 \gamma(0) = \pi_M(e_1) ,\gamma(1) = \pi_M(e_2) \Bigr\},
\end{multline*}
where \(P^\gamma \colon \pi^{-1}_M (\{\gamma (1)\}) \to \pi^{-1}_M (\{\gamma (0)\}) \) is the parallel transport along \(\gamma\) with respect to the connection \(K_E\). 
\end{propo}

The \emph{Cheeger--Gromoll metric} is defined for every \(\nu \in TE\) by
\begin{align*}
G_E^{CG} (\nu) = g_M (D\pi_M (\nu)) + \frac{g_E (K_E (\nu)) + (g_E (K_{E} (\nu), \pi_E (\nu)))^2}{1+g_E (\pi_E(\nu))}.
\end{align*}

%CG metric
\begin{propo}\label{propo_cheeger}
The Cheeger--Gromoll metric \(G_E^{CG}\) is strongly concordant with \(g_E\).
\end{propo}

\begin{proof}
We have
\[
  G_E^{CG} (D \pi_M (\nu)) 
  = g_M (D \pi_M (\nu)) \le G_E^S (\nu).
\]
Let \(\nu \in TE\), we have since the connection \(K_E\) is metric
\[
\begin{split}
  (D g_E (\nu))^2 & = 4 (g_E (K_E (\nu), \pi_E (\nu)))^2\\
  & \le 4 \Bigl(\frac{g_E (K_E (\nu)) g_E (\pi_E(\nu))}{1 + g_E (\pi_E (\nu))}
  + \frac{(g_E (K_E (\nu), \pi_E (\nu)))^2g_E (\pi_M(\nu))}{1 + g_E (\pi_E (\nu))}\Bigr)\\
  & \le 4 G_E^{CG} (\nu) g_E (\pi_E(\nu)).
\end{split}
\]
Finally, since \(\pi_M (\Vertlift_e (e)) = e\), \(D \pi_M (\Vertlift_e(e)) = 0\) and \(K_E (\Vertlift_e (e)) = e\), we obtain
\[
  G_E^{CG} (\Vertlift_e(e)) = \frac{g_E (e) + g_E (e)^2}{1 + g_E (e)} = g_E (e).\qedhere
\]
\end{proof}

%COMPARISON SASAKI -- CHEEGER GROMOLL
Finally we have
\begin{lemme}[Comparison between the Cheeger--Gromoll and Sasaki metrics]
\label{lemmaComparisonCGS}
For every \(\nu \in TE\),
\[
  G_E^{CG} (\nu) \le G_E^{S} (\nu).
\]
\end{lemme}
\begin{proof}
We observe that by the Cauchy--Schwarz inequality
\[
  (g_E (K_{E} (\nu), \pi_M (\nu)))^2
   \le g_E (K_E (\nu)) g_E (\pi_M (\nu)).\qedhere 
\]
\end{proof}

The Sasaki and Cheeger--Gromoll metrics are particular cases for \(\lambda = 1\) and \(\lambda =\frac{1}{2}\) of metrics
\begin{equation}\label{eqLambda}
G_E^\lambda (\nu) = g_M (D\pi_M (\nu))
+ \frac{\lambda g_E (K_E (\nu)) + (1 - \lambda)(g_E (K_{E} (\nu), \pi_M (\nu)))^2}{\lambda + (1 - \lambda) g_E (\pi_M (\nu))}.
\end{equation}
In the limit case \(\lambda = 0\), the metric
\[
 G_E^0 (\nu) = g_M (D\pi_M (\nu))
+ \frac{(g_E (K_{E} (\nu), \pi_M (\nu)))^2}{g_E (\pi_M (\nu))} 
\]
is degenerate; the associated degenerate distance is 
\[
 d_E (e_1, e_2)^2 = d_M (\pi (e_1), \pi_M(e_2))^2 + \abs{\abs{e_1}_{g_E} - \abs{e_2}_{g_E}}^2.
\]

In our study of Sobolev spaces, we will be interested in metrics on the bundle \((T^*M \otimes TN, \pi_{M\times N}, M \times N)\). 
Assume that the Riemannian manifolds are of class \(C^2\) and that we have metric connections \(K_{T^*M}\) and \(K_{TN}\) on \(T^*M\) and \(TN\). 
Such connections are given by the Levi-Civita connection.
We can then define for \(v \in T(T^*M)\) and \(w \in T(TN)\),
\[ 
  K_{T^*M \otimes TN}(v \otimes w) = K_{T^*M}(v) \otimes K_{TN}(w).
\]
Moreover \(K_{T^*M \otimes TN}\) is metric with \(\gF\).

%APPLICATION SOBOLEV SPACES
Given a concordant distance \(\dTN\), we define the metric \(\delta_{1, p}\) for \(u, v \in \dot{W}^{1, p} (M, N)\) by
\[
  \dW(u, v) = \left(\int_M \dTN(Du,Dv)^p \, d\mu \right)^\frac{1}{p} \in [0, \infty].
\]
This distance can be infinite. This will happen for instance if \(M\) has infinite Riemannian volume and \(u, v\) are distinct constant maps.

We first prove that concordant distances characterize Sobolev maps. Given two maps \(u, v : M \to N\) we denote for every \(x \in M\),
\[
 d (u, v) (x) = d((x, u(x)), (x, v (x))),
\]
with the usual identification of \(M \times N\) with \(M \times N \times \{0\} \subset T^*M \otimes TN\).

\begin{propo}\label{SobolevChara}
Assume that \(\dTN\) is concordant with \(\abs{\cdot}_{\gF}\). If \(v\in \Wp\) then \(u\in \Wp\) and \(d (u, v) \in \Lr\) if and only if the map \(u\) is colocally weakly differentiable and \(\dTN(Du,Dv) \in \Lr\).
\end{propo}

\begin{proof}
The first assumptions follows from the fact that 
\[ 
\begin{split}
\dTN(Du,Dv) &\le \dTN (Du, u) + \dTN (u, v) + \dTN (v, Dv)\\
& \le d (u, v) + \kappa (\abs{Du}_{\gF} + \abs{Dv}_{\gF})
\end{split}
\]
almost everywhere in \(M\).
Conversely we observe that, by the nonexpansiveness property
\[
  d (u, v) \le \dTN (Du, Dv)
\]
almost everywhere in \(M\)
and that by the equivalence assumption
\[ 
\begin{split}
  \abs{Du}_{\gF} & \le \kappa \dTN (Du, u)
  \le \kappa \bigl(\dTN (Du, Dv) + \dTN (Dv, v)\bigr)\\
  & \le \kappa \dTN (Du, Dv) + \kappa^2 \abs{Dv}_{\gF},
\end{split}
\]
almost everywhere in \(M\); the converse statement follows.
\end{proof}

%SAME TOPO
\begin{propo}\label{sob_equi}
Assume that \(\dTN\) is concordant with \(\abs{\cdot}_{\gF}\). If \((u_\ell)_{\ell \in \N}\) is a sequence in \(\Wp\) and \( u \in \dot{W}^{1,p}(M,N)\) then  the sequence \((\dTN(Du_\ell,Du))_{\ell \in \N}\) converges to \(0\) in \(\Lr\) if and only if the sequence \((u_\ell)_{\ell \in \N}\) converges strongly to \(u\) in \(\Wp\) and the sequence \((d (u_\ell, u))_{\ell \in \N}\) converges to \(0\) in \(\Lr\).
\end{propo}

\begin{proof}
First assume that \((u_\ell)_{\ell \in \N}\) converges strongly to \(u\) in \(\Wp\) and \((d(u_\ell, u))_{\ell \in \N}\) converges to \(0\) in \(\Lr\). By the definition of convergence in \(\Wp\) (definition~\ref{def_convergence_sobolev}), 
\((\abs{D u_\ell}_{\gF})_{\ell \in \N}\) converges to \(\abs{D u}_{\gF}\) in \(\Lr\)
and \((D u_\ell)_{\ell \in \N}\) converges to \(Du\) locally in measure.
The latter convergence implies that \((\dTN (Du_\ell, Du))_{\ell \in \N}\) converges to \(0\) in measure. Since for every \(\ell \in \N\),
\[
  \dTN(Du_\ell, Du) \le d (u_\ell, u) + \kappa (\abs{Du_\ell}_{\gF} + \abs{Du}_{\gF})
\]
almost everywhere in \(M\), 
the conclusion follows from Lebesgue's dominated convergence theorem \cite{bogachev}*{theorem 2.8.5}.

Conversely, if \((\dTN(Du_\ell,Du))_{\ell \in \N}\) converges to \(0\) in \(\Lr\),
then \((\dTN(Du_\ell,Du))_{\ell \in \N}\) converges to \(0\) in measure and thus \((D u_\ell)_{\ell \in \N}\) converges to \(D u\) locally in measure.
Moreover, \((\abs{D u_\ell}_{\gF})_{\ell \in \N}\) converges to \(\abs{D u}_{\gF}\) in measure. Since for every \(\ell \in \N\),
\[
  \abs{D u_\ell}_{\gF} \le 	\kappa \dTN (Du_\ell, Du) + \kappa^2 \abs{D u}_{\gF}
\]
almost everywhere in \(M\), 
by Lebesgue's dominated convergence theorem, the sequence \((\abs{D u_\ell}_{\gF})_{\ell \in \N}\) converges to \(\abs{D u}_{\gF}\) in \(\Lr\).
Finally, since \(d (u_\ell, u) \le \dTN (Du_\ell, Du)\), it is clear that the sequence \((d (u_\ell, u))_{\ell \in \N}\) converges to \(0\) in \(\Lr\).
\end{proof}

Finally, Sobolev spaces are complete under all these metrics.

%THM SOBOLEV COMPLETE
\begin{propo}\label{complete}
If for almost every \(x \in M\), the metric space \((\{x\} \times N, d)\) is complete and \(d\) is concordant with \(\abs{\cdot}_{\gF}\), then the Sobolev space \((\Wp, \dW)\) is complete. 
\end{propo}
\begin{proof}
Let \((u_\ell)_{\ell \in \N}\) be a Cauchy sequence for the metric \(\dW\).
There exists a subsequence \((D u_{\ell_k})_{k \in \N}\) which is a Cauchy sequence for \(d\) almost everywhere in \(M\). 
By the nonexpansiveness of the projection, \((u_{\ell_k})_{k \in \N}\) is a Cauchy sequence for \(d\) almost everywhere in \(M\). Since for almost every \(x \in M\), the metric space \((\{x\} \times N, d)\) is complete, the sequence \((u_{\ell_k})_{k\in \N}\) converges almost everywhere to a map \(u : M \to N\).

Since the distance \(d\) is concordant, for every \(k \in \N\),
\[
 \abs{D u_{\ell_k}}_{\gF} \le \kappa d (D u_{\ell_k}, D u_{\ell_0}) + \kappa^2 \abs{D u_{\ell_0}}_{\gF}
\]
almost everywhere in \(M\) and we deduce that the sequence \((\abs{D u_{\ell_k}})_{k \in \N}\) is bounded almost everywhere in \(M\) and 
\[
 \limsup_{k \to \infty} \int_{M} \abs{D u_{\ell_k}}_{\gF}^p < \infty.
\]
Since the distance \(d\) is complete on any compact subset of \(T^*M \otimes TN\), the sequence \((D u_{\ell_k})_{k \in \N}\)
converges almost everywhere to a bundle morphism \(\upsilon : TM \to TN\). 
By the closure property (proposition~\ref{prop_closure}), we deduce that \(u \in \dot{W}^{1,p}(M,N)\) and \(D u = \upsilon\).
By Fatou's lemma, for every \(k \in \N\),
\[
 \delta_{1, p} (u_{\ell_k}, u) \le \liminf_{j \to \infty} \delta_{1, p} (u_{\ell_k}, u_{\ell_j}),
\]
and thus 
\[
  \lim_{k \to \infty} \delta_{1, p} (u_{\ell_k}, u) = 0.
\]
Since the sequence \((u_\ell)_{\ell \in \N}\) is a Cauchy sequence for \(\delta_{1, p}\), the conclusion follows.
\end{proof}

% COMPARISON METRICS ON WP
If we consider a Sobolev space \(\dot{W}^{1, p} (M, N)\) we have various distances under the hand:
if \( \iota \colon N \to \R^\nu\) is an  isometric embedding and \(\iota (N)\) is closed, the Chiron distance, 
\[
 \delta_{1, p}^C (u, v) = \Bigl(\int_{M} \bigl(\abs{\iota \circ u - \iota \circ v}^2 + \abs{\abs{D(\iota \circ u)} - \abs{D(\iota \circ v)}}^2 \bigr)^{\frac{p}{2}} \Bigr)^\frac{1}{p}
\]
if \(M\) and \(N\) are of class \(C^2\), the Sasaki distance 
\[
  \delta_{1, p}^S (u, v) = \Bigl(\int_{M} (d_{T^* M \otimes TN}^S (Du, Dv))^p\Bigr)^\frac{1}{p},
\]
the Cheeger--Gromoll distance
\[
  \delta_{1, p}^{CG} (u, v) = \Bigl(\int_{M} (d_{T^* M \otimes TN}^{CG} (Du, Dv))^p  \Bigr)^\frac{1}{p},
\]
and the induced distance,
\[
  \delta_{1, p}^\iota (u, v) = \Bigl(\int_{M} \bigl(\abs{\iota \circ u - \iota \circ v}^2 + \abs{D(\iota \circ u) - D(\iota \circ v)}^2 \bigr)^{\frac{p}{2}}\Bigr)^\frac{1}{p}.
\]
(The distance on the right is in fact the Sasaki distance on \(T^*M \otimes T\R^\nu\).)
It follows from the previous results that the Sobolev space \(\dot{W}^{1, p} (M, N)\) has the same topology for all these distances and is complete except for the Chiron distance. 

By strong concordance of all the metrics under hand, we have 
\[
 \delta_{1, p}^C (u, v) \le \sqrt{2} \min(\delta_{1, p}^S (u, v), \delta_{1, p}^{CG} (u, v), \delta_{1, p}^\iota (u, v) ).
\]
That is the identity map from \(\dot{W}^{1,p}(M,N)\) endowed with any of the distance in the right-hand side is a Lipschitz map into \((\dot{W}^{1,p}(M,N), \delta_{1, p}^C)\). Since the convergences are equivalent (proposition~\ref{corChironEquiv}), the inverse of the identity is not uniformly continuous because \((\dot{W}^{1,p}(M, N), \delta_{1,p}^C)\) is not complete.

\begin{propo}
If \(M\) and \(N\) are nonempty Riemannian manifold and  \(p\in [1,\infty[\), then
\[ 
  (\dot{W}^{1,p}(M, N), \delta_{1,p}^C)
\]
is not complete. 
\end{propo}
\begin{proof}
We give the proof when \(M = (0, 1)\). The reader will adapt the proof to the general case.
Choose \(\gamma \in C^1 ([0,L], N)\) such that for every \(t \in [0, L]\), \(\abs{\gamma' (t)}_{g_N} = 1\) and define, following Chiron \cite{chiron}*{lemma 2}, for every \(\ell \in \N\) the function \(u_\ell : (0, 1) \to N\) for each \(t \in (0, 1)\) by
\begin{equation*}
u_\ell(t) =
\gamma (\operatorname{dist}(t, \mathbb{Z}/n)).
\end{equation*}
For every \(\ell \in \N\), the function \(u_\ell\) is Lipschitz and \( \abs{{u'}_\ell}=1\) almost everywhere in \((0,1)\).
Moreover, the sequence \((u_\ell)_{\ell \in \N}\) converges uniformly to the constant map \(u = \gamma (0)\). 
Since 
\[
  \lim_{\ell \to \infty} \int_0^1 \abs{u_\ell'}^p = 1 \ne 0 = \int_0^1 \abs{u'}^p,
\]
the sequence \((u_\ell)_{\ell \in \N}\) does not converge in \(\dot{W}^{1,p}((0, 1), N)\).
By proposition~\ref{corChironEquiv}, the space \((\dot{W}^{1,p}(M, N), \delta_{1,p}^C)\) is not complete.
\end{proof}

We also always have, in view of the comparison between the Cheeger--Gromoll and Sasaki metrics (lemma~\ref{lemmaComparisonCGS})
\[
  \delta_{1, p}^{CG} (u, v) \le \delta_{1, p}^S (u, v),
\]
that is, the identity is nonexpansive from \((\dot{W}^{1, p} (M, N), \delta_{1, p}^{S})\) to \((\dot{W}^{1, p} (M, N), \delta_{1, p}^{CG})\).
On the other hand, we observe that \(\delta_{1, p}^{S}\) is not uniformly equivalent with \(\delta_{1, p}^{CG}\).

\begin{propo}
\label{not_equivalent_0}
If \(M\) is a nonempty Riemannian manifold \(p \in [1,\infty[\) and \(n\ge 2\), then the identity map 
\[
  i \colon (\dot{W}^{1,p}(M, {\R^n}), \delta_{1,p}^{CG}) \to (\dot{W}^{1,p}(M,\R^n), \delta_{1, p}^{S})
\]
is not uniformly continuous.
\end{propo}
\begin{proof}
First we consider the case where \(M = \R\). 
We choose a map \(u \in C^1_c (\R, \R^n) \setminus \{0\}\) such that \(\supp u \subset (-1, 1)\), and we define for \(\lambda > 0\) the maps \(u_\lambda : \R \to \R^n\) and \(v_\lambda : \R \to \R^n\) for each \(t \in \R\) by \(u_\lambda (t) = u (t/\lambda)\) and \(v_\lambda (t)= - u_\lambda (t) = -u (t / \lambda)\).
We have for every \(t \in \R\), since \(n \ge 2\) and \(\abs{u_\lambda' (t)} = \abs{v_\lambda' (t)}\), 
\[
  d^{CG} \bigl(u_\lambda' (t), v_\lambda' (t)\bigr)^2
  \le 4\abs{u_\lambda (t)}^2 + \frac{\abs{u_\lambda' (t)}^2\pi^2}{1 + \abs{u_\lambda' (t)}^2} \le 4\abs{u_\lambda (t)}^2 + \pi^2.
\]
Therefore, for every \(\lambda > 0\),
\[
  \delta_{1, p}^{CG} (u_\lambda, v_\lambda) =
  \Bigl(\int_{-\lambda}^{\lambda} d^{CG}_{T^* \R \otimes T\R^n} (u_\lambda', v_\lambda')^p\Bigr)^\frac{1}{p}
  \le \sqrt{4\norm{u}_{L^\infty}^2 +\pi^2} (2\lambda)^\frac{1}{p}.
\]
On the other hand, since \(\R^n\) is flat, we have for each \(\lambda > 0\),
\[
\begin{split}
  \delta_{1, p}^{S} (u_\lambda, v_\lambda) =
  \Bigl(\int_{-\lambda}^{\lambda} d^{S}_{T^* \R \otimes T\R^n} (u_\lambda', v_\lambda')^p \Bigr)^\frac{1}{p}
  &= \Bigl(\int_{-\lambda}^{\lambda} \bigl( \abs{u_\lambda - v_\lambda}^2 + \abs{u_\lambda' - v_\lambda'}^2\bigr)^\frac{p}{2} \Bigr)^\frac{1}{p}  \\
  &=2\Bigl(\int_{-1}^{1} \bigl(\lambda^\frac{2}{p} \abs{u}^2 + \lambda^{-2 \bigl(1 - \frac{1}{p}\bigr)} \abs{u'}^2\bigr)^\frac{p}{2} \Bigr)^\frac{1}{p}.
\end{split}
\]
We have thus 
\begin{align*}
  \lim_{\lambda \to 0} \delta_{1, p}^{CG} (u_\lambda, v_\lambda) &= 0 &
   & \text{ and } &
   \lim_{\lambda \to 0} \delta_{1, p}^{S} (u_\lambda, v_\lambda) &> 0,
\end{align*}
which shows the non uniform continuity.

In the general case, let \(a \in M\) and let \(\rho \in (0, \rho_i(a))\) where \(\rho_i (a)\) is the  injectivity radius of the Riemannian manifold \(M\) at the point \(a\). We define then the maps \(u_\lambda : M \to \R^n\) and \(v_\lambda : M \to \R^n\) for \(\lambda > 0\) and \(x \in M\) by 
\[
  u_\lambda (x) = u \Bigl(\frac{d (x, a) - \rho}{\lambda} \Bigr)
\]
and
\[
  v_\lambda (x) = - u \Bigl(\frac{d (x, a) - \rho}{\lambda} \Bigr);
\]
the non uniform continuity follows as in the case \(M = \R\) treated above.
\end{proof}

Finally we show that the Sasaki distance \(\delta_{1,p}^S\) and the embedding distance \(\delta_{1, p}^\iota\) are not uniformly comparable distances in general.

%EXAMPLE WITH S^n
\begin{propo}\label{not_equivalent}
If \(M\) is a nonempty Riemannian manifold, \(p \in [1,\infty[\) and \(n\ge 2\), then the identity map 
\[
  i \colon (\dot{W}^{1,p}(M, {\mathbb{S}^n}), \delta_{1,p}^S) \to (\dot{W}^{1,p}(M,\mathbb{S}^n), \delta_{1, p}^\iota)
\]
is not uniformly continuous.
\end{propo}

\begin{proof}
We begin by considering the case \(M = \R\).
Choose \(y \in \mathbb{S}^n\) and \(\rho \colon \mathbb{S}^n \to \mathbb{S}^n\) be a non-identical isometry such that \(\rho(y) = y\) and \(u \in C^1 (\R, \mathbb{S}^n)\) such that for every \(x\in \R \setminus [-1,1]\), \(\rho (u(x))= y\) and \(\rho (u (x))\ne u(x)\) in \((-1, 1)\). 

Let \(u_\lambda : \R \to \mathbb{S}^n\) be defined for every \(t \in \R\) by \(u_\lambda (t) = u(t/\lambda)\) and let \(v_\lambda = \rho \circ u_\lambda\).
Since \(\rho\) is an isometry, \(D\rho \colon T\mathbb{S}^n \to T\mathbb{S}^n\) is an isometry on tangent vectors. 
Since \(n \ge 2\), for every \(e \in T\mathbb{S}^n\) there exists a path \(\gamma\) such that \(P^\gamma (e) = D \rho (e)\); the length of \(\gamma\) with respect to the Sasaki metric \(G^S\) can be bounded uniformly by \(2 \pi\).
Therefore, we have, for every \(e \in T\mathbb{S}^n\), by proposition~\ref{lemme_sas_equi}, 
\[
  d_{T\mathbb{S}^n}^S (e, D \rho(e)) \le 2 \pi.
\]
We deduce that 
\begin{align*}
  \delta_{1,p}^S(u_\lambda, v_\lambda) & = \Bigl(\int_{\R} d^S_{T^* \R\otimes T\mathbb{S}^n}(D u_\lambda, D v_\lambda)^p\Bigr)^\frac{1}{p}\le 2 \pi\, \lambda^\frac{1}{p}.
\end{align*}
On the other hand, 
\[
\delta_{1, p}^\iota (u_\lambda, v_\lambda)
= \Bigl(\int_{\R} \lambda \bigl(\abs{u - \rho \circ u}^2 + \lambda^{-2} \abs{D u - D (\rho \circ u)}^2 \bigr)^\frac{p}{2}\Bigr)^\frac{1}{p},
\]
Consequently, \(\lim_{\lambda \to 0} \delta_{1,p}^S(u_\lambda, v_\lambda) =0\)
while \(\liminf_{\lambda \to 0} \delta_{1, p}^\iota (u_\lambda, v_\lambda) = \infty\).

In the general case, we conclude as in the proof of proposition~\ref{not_equivalent_0}.
\end{proof}

\begin{propo}
If \(M\) is a nonempty Riemannian manifold and \(p\in [1, \dim M[\), then the identity embedding map 
\[
  i \colon(\dot{W}^{1,p}(M,\mathbb{S}^1), \delta_{1, p}^\iota) \to (\dot{W}^{1,p}(M, \mathbb{S}^1), \delta_{1,p}^S).
\]
is not uniformly continuous.
\end{propo}
\begin{proof}
Fix \(y \in M\) and define
\[
  \varphi_\lambda (x) =
  \begin{cases}
    \frac{\pi}{2} & \text{if \(d (x, y) \ge 2\lambda\)},\\
    \frac{d (x, y) - \lambda}{2\lambda} \pi & \text{if \(\lambda \le d (x, y) < 2\lambda\)},\\
    \lambda \sin \frac{(d (x, y) - \lambda)\pi}{2\lambda^{1 + \frac{\dim M}{p}}} & \text{if \(d (x, y) < \lambda \)}.
  \end{cases}
\]
We define now the maps \(u_\lambda = (\cos (\varphi_\lambda), \sin (\varphi_{\lambda}))\) and \(v_\lambda = ({-\cos (\varphi_\lambda)}, \sin (\varphi_{\lambda}))\). 
We observe that for every \(\lambda > 0\)
\[
  \delta_{1, p}^{S} (u_\lambda, v_\lambda)
  =2\Bigl(\int_{M} \bigl(\cos (\varphi_\lambda)^2 + \abs{D \varphi_\lambda}^2\bigr)^{\frac{p}{2}}\Bigr)^\frac{1}{p}
\]
and thus 
\[
  \lim_{\lambda \to 0} \delta_{1, p}^{S} (u_\lambda, v_\lambda)
  > 0.
\]
On the other hand, since for each \(\lambda > 0\),
\[
\begin{split}
  \delta_{1, p}^{\iota} (u_\lambda, v_\lambda)
  &=2\Bigl(\int_{M} \bigl(\cos (\varphi_\lambda)^2 + \sin (\varphi_\lambda)^2 \abs{D \varphi_\lambda}^2\bigr)^{\frac{p}{2}}\Bigr)^\frac{1}{p}\\
  &\le 2\Bigl(\int_{M} \bigl(\cos (\varphi_\lambda)^2 + \abs{\varphi_\lambda}^2 \abs{D \varphi_\lambda}^2\bigr)^{\frac{p}{2}}\Bigr)^\frac{1}{p}.
\end{split}
\]
we have therefore, since \(p < \dim M\),
\[
  \lim_{\lambda \to 0} \delta^{\iota}_{1, p} (u_\lambda, v_\lambda) = 0.\qedhere
\]
\end{proof}

%-------------------------BIBLIO

\end{document}